\documentclass[12pt]{article}

\usepackage{epsfig}

\usepackage{caption}
\usepackage{amsfonts}

\usepackage{amssymb}
\usepackage{mathrsfs}
\usepackage{amsmath} 

\usepackage{graphicx} 
\usepackage{float}
\usepackage{pict2e}
\usepackage{enumerate}
\usepackage{enumitem}
\usepackage{authblk} 

\usepackage{graphics} 
\usepackage{tikz} 
\usepackage{ulem} 
\usepackage{cases}
\usetikzlibrary{calc}

\usepackage[f]{esvect}
\newcommand{\cevv}[1]{\reflectbox{\ensuremath{\vv{\reflectbox{\ensuremath{#1}}}}}}

\usepackage{tikz-3dplot}
\usetikzlibrary{hobby}

\usepackage{subcaption}

\usetikzlibrary{shapes.geometric}
\usetikzlibrary{decorations.pathmorphing,calligraphy}

\numberwithin{equation}{section}

\usepackage[pdfauthor={derajan},pdftitle={How to do this},pdfstartview=XYZ,bookmarks=true,
colorlinks=true,linkcolor=blue,urlcolor=blue,citecolor=blue,pdftex,bookmarks=true,linktocpage=true,hyperindex=true]{hyperref}

\usepackage{subdepth}

\usepackage{color}

\definecolor{darkgreen}{RGB}{0,154,23}
\definecolor{darkbrown}{RGB}{183,60,18}

\newcommand{\C}{\mathcal{C}}

\usepackage[colorinlistoftodos,prependcaption,textsize=scriptsize, textwidth=24mm,]{todonotes}

\usetikzlibrary{patterns}

\makeatother

\begin{document} 
	\newtheorem{theorem}{Theorem}
	\newtheorem{observation}[theorem]{Observation}
	\newtheorem{corollary}[theorem]{Corollary}
	\newtheorem{algorithm}[theorem]{Algorithm}
	\newtheorem{definition}[theorem]{Definition}
	\newtheorem{guess}[theorem]{Conjecture}
	
	\newtheorem{problem}[theorem]{Problem}
	\newtheorem{question}[theorem]{Question}
	\newtheorem{lemma}[theorem]{Lemma}
	\newtheorem{proposition}[theorem]{Proposition}
	\newtheorem{fact}[theorem]{Fact}
	
	\newtheorem{claim}{Claim} 
	\newtheorem{claimA}{Claim}
	\renewcommand{\theclaimA}{1\alph{claimA}}
	\newtheorem{claimB}{Claim} 
	\renewcommand{\theclaimB}{2\alph{claimB}}
	\newtheorem{claimC}{Claim} 
	\renewcommand{\theclaimC}{3\alph{claimC}}
	
	\captionsetup[figure]{labelfont={bf},name={Fig.},labelsep=period}
	
	\makeatletter
	\newcommand\figcaption{\def\@captype{figure}\caption}
	\newcommand\tabcaption{\def\@captype{table}\caption}
	\makeatother

	\newtheorem{acknowledgement}[theorem]{Acknowledgement}
	
	\newtheorem{axiom}[theorem]{Axiom}
	\newtheorem{case}[theorem]{Case}
	\newtheorem{conclusion}[theorem]{Conclusion}
	\newtheorem{condition}[theorem]{Condition}
	\newtheorem{conjecture}[theorem]{Conjecture}
	\newtheorem{criterion}[theorem]{Criterion}
	\newtheorem{example}[theorem]{Example}
	\newtheorem{exercise}[theorem]{Exercise}
	\newtheorem{notation}{Notation} 
	\newtheorem{solution}[theorem]{Solution}
	\newtheorem{summary}[theorem]{Summary} 
	
	\newenvironment{proof}{\noindent {\bf
			Proof.}}{\rule{2.5mm}{2.5mm}\par\medskip}
	\newcommand{\remark}{\medskip\par\noindent {\bf Remark.~~}}
	
	\newcommand{\qed}{\null\nobreak\rule{0.6em}{0.6em}}

	\let\svthefootnote\thefootnote
	\newcommand\freefootnote[1]{%
		\let\thefootnote\relax%
		\footnotetext{#1}%
		\let\thefootnote\svthefootnote%
	}
	
	\title{\large{\bf Proof of a conjecture of Voss on bridges of longest cycles}}

	\author{Jie Ma\thanks{School of Mathematical Sciences, University of Science and Technology of China, Hefei, Anhui 230026, and Yau Mathematical Sciences Center, Tsinghua University, Beijing 100084, China. Research supported by National Key Research and Development Program of China 2023YFA1010201 and National Natural Science Foundation of China grant 12125106. Email: {\tt jiema@ustc.edu.cn}.}  \quad \quad Rongxing Xu\thanks{School of Mathematical Sciences, Zhejiang Normal University, Jinhua, Zhejiang, 321000, China. Research supported by National Natural Science Foundation for Young Scientists of China grant 12401472. Email: {\tt xurongxing@zjnu.edu.cn}.}}

	\maketitle
	
	\begin{abstract}
		Bridges are a classical concept in structural graph theory and play a fundamental role in the study of cycles. A conjecture of Voss from 1991 asserts that if disjoint bridges $B_1, B_2, \ldots, B_k$ of a longest cycle $L$ in a $2$-connected graph overlap in a tree-like manner (i.e., induce a tree in the {\it overlap graph} of $L$), then the total {\it length} of these bridges is at most half the length of $L$. Voss established this for $k \leq 3$ and used it as a key tool in his 1991 monograph on cycles and bridges. 
		In this paper, we confirm the conjecture in full via a reduction to a cycle covering problem.
	\end{abstract} 
	
	\section{Introduction}
	Let $G$ be a graph and $H$ a subgraph of $G$.  
	An \emph{$H$-bridge} of $G$ is either  
	(i) an edge in $E(G) \setminus E(H)$ with both endpoints in $V(H)$, or  
	(ii) a subgraph consisting of a component $D$ of $G - V(H)$ together with all edges between $V(D)$ and $V(H)$.
	For an $H$-bridge $B$, the vertices in $V(H) \cap V(B)$ are called the \emph{attachments} of $B$.
	In this paper, we often consider $H$ to be a cycle.

	The concept of bridges has naturally emerged in the development of graph theory, particularly in the study of cycles. 
	As emphasized by Bondy in his influential survey~\cite{GGL} (p.~58), ``bridges clearly play a very important role in the study of paths and circuits, and it can be argued that their role is central.''  
	
	A cornerstone result in the study of cycles is Tutte's theorem~\cite{Tutte1956}, which strengthens Whitney’s theorem~\cite{Whitney1931} by asserting that every $4$-connected planar graph is Hamiltonian.  
	A key ingredient in Tutte's celebrated proof is the so-called Bridge Lemma, which characterizes the bridges of certain cycles in planar graphs by their attachments.  
	Since then, bridges have facilitated numerous generalizations and refinements concerning Hamiltonicity, including results of Thomassen~\cite{Thomassen1983} (a small omission was corrected by Chiba and Nishizeki \cite{CN1986}), Thomas and Yu~\cite{TY1994,TY1997}, Kawarabayashi and Ozeki~\cite{KO2015}, as well as \cite{JY2002,OZ2018,Sanders1997,TYZ2005}.  
	Beyond Hamiltonicity, bridges have also played a central role in the study of longest cycles in general graphs, as explored in~\cite{BE1980,CSYZ2006,CY2002,CYZ2012,JW1992,WY2023}, among others.
	
	In light of these advances, an important direction of research has been to understand how cycles interact with their bridges and how these bridges are arranged along the cycle. 
	This leads to the notion of the \emph{overlap graph}. 
	Let $L$ be a cycle in a graph $G$. 
	Two $L$-bridges $B_1$ and $B_2$ are said to \emph{overlap} if $L$ cannot be partitioned into two subpaths $L_1$ and $L_2$ such that the attachments of $B_i$ lie entirely on $L_i$ for each $i=1,2$. 
	The \emph{overlap graph} of $G$ with respect to $L$, denoted $O_G(L)$, is the graph whose vertices correspond to the $L$-bridges, with an edge between two vertices if and only if the corresponding bridges overlap.
	Using this concept, Tutte~\cite{Tutte1959} gave a characterization of planar graphs, proving that a graph $G$ is planar if and only if, for every cycle $L$, the overlap graph $O_G(L)$ is bipartite. This characterization underlies most planarity-testing algorithms (see \cite{AP1961,Goldstein1963,HT1974}). 
	Voss~\cite{Voss1991} further showed that for any cycle $L$ in a $3$-connected graph $G$, the overlap graph $O_G(L)$ is connected.
	
	In his monograph~\cite{Voss1991}, Voss investigated various problems on cycles, with particular emphasis on the role of bridges.  
	To provide a measure on the size of a bridge, he~\cite{Voss1991} introduced the following parameter (which is called the span of a bridge in~\cite{GGL}).
	
	\begin{definition}\label{def-length}
		For a subgraph $H$ in a graph $G$, the {\bf length} $\lambda(B)$ of an $H$-bridge $B$ is the maximum number of edges in a tree within $B$ whose leaves are exactly the attachments of $B$.
	\end{definition}
	
	\noindent
	For a cycle $L$, an $L$-bridge has length one if and only if it is a chord of $L$.  
	Thomassen’s Chord Conjecture (see, e.g., \cite{AG1985,Bondy2014,Thomassen1989}) then asserts that every longest cycle $L$ in a $3$-connected graph contains such a bridge.  
	The conjecture remains open, with significant progress in~\cite{BM2008,KNZ2007,Thomassen1997,Zhang1987}.
	
	Voss~\cite{Voss1991} proposed the following conjecture on longest cycles $L$ in a graph $G$, aiming to provide quantitative control over the size of $L$-bridges relative to the length of $L$ (i.e., the {\it circumference} of $G$).  
	The conjecture is also discussed in Bondy’s comprehensive survey on cycles~\cite{GGL} (see Conjecture~5.11).

	\begin{conjecture}[Voss \cite{Voss1991}, p.~54]\label{main-conj}
		Let $G$ be a $2$-connected graph with a longest cycle $L$. 
		Let $B_1$, $B_2, \ldots, B_k$ be $L$-bridges that are pairwise vertex-disjoint and induce a tree in $O_G(L)$. Then
		$$\sum_{i=1}^{k} \lambda(B_i) \leq \lfloor |E(L)|/2 \rfloor.$$
	\end{conjecture}

	The cases $k \leq 3$ were established by Voss himself~\cite{Voss1991}, who used them as key tools in his study of problems and properties of longest cycles (see Chapters 3, 7, and 11 of~\cite{Voss1991}).
	
	We note that, if true, the inequality is best possible.  
	For $k=1$, consider $G=K_{2,3}$, where the longest cycle has length $4$ and its unique bridge has length $2$.  
	For $k \geq 2$, let $G$ be the graph obtained from a $2k$-cycle $L = v_1v_2\cdots v_{2k}v_1$ by adding the edge $v_1v_{k+1}$ and, for each $i \in \{2,3,\ldots,k\}$, the edge $v_iv_{2k+2-i}$. 
	Then $G$ has exactly $k$ $L$-bridges whose total length is $k$, which equals half of $|E(L)|=2k$.
	We also remark that the conjecture fails if the subgraph of $O_G(L)$ induced by $B_1,B_2,\ldots,B_k$ is disconnected or contains a cycle, as shown in Fig.~\ref{fig:2examples}.  
	To be precise, we see that in the left graph, $O_G(L)$ is disconnected and $\lambda(B_1)+\lambda(B_2)=4 > |E(L)|/2=3$, while in the right graph (the Petersen graph), $O_G(L)$ contains a triangle and $\lambda(B_1)+\lambda(B_2)+\lambda(B_3)=5 > |E(L)|/2=9/2$.
	
	\begin{figure}[H]
		\begin{minipage}[t]{0.49\textwidth}
			\centering
			\begin{tikzpicture}[>=latex,
				roundnode/.style={circle, draw=black, fill=white, minimum size=1.5mm, inner sep=0pt}]
				\draw ($(0,0)+(0:2 and 1.3)$) arc (0:360:2 and 1.3); 
				\node [roundnode] (A1) at ($(0,0)+(75:2 and 1.3)$){}; 
				\node [roundnode] (A2) at ($(0,0)+(105:2 and 1.3)$){};
				\node [roundnode] (A3) at ($(0,0)+(180:2 and 1.3)$){};
				\node [roundnode] (A4) at ($(0,0)+(255:2 and 1.3)$){};
				\node [roundnode] (A5) at ($(0,0)+(285:2 and 1.3)$){};
				\node [roundnode] (A6) at ($(0,0)+(0:2 and 1.3)$){};
				\draw(A2)--(A4);
				\draw(A1)--(A5);
				\node [roundnode] at (-0.51,0) {};
				\node at (-1,0) {$B_1$};
				\node [roundnode] at (0.51,0) {};
				\node at (1,0) {$B_2$};
				\node at (2.4,0) {$G$};
				\node at (1.7,1) {$L$};
			\end{tikzpicture} 
		\end{minipage}  
		\begin{minipage}[t]{0.49\textwidth}
			\centering
			\begin{tikzpicture}[>=latex,	
				blacknode/.style={circle, draw=black, fill=white, minimum size=1.5mm, inner sep=1pt}] 
				\draw (0,0) circle(1.5cm);
				\node [blacknode] (A1) at (90:1.5cm){};
				\node [blacknode] (A2) at (130:1.5cm){};
				\node [blacknode] (A3) at (170:1.5cm){};
				\node [blacknode] (A4) at (210:1.5cm){};
				\node [blacknode] (A5) at (250:1.5cm){};
				\node [blacknode] (A6) at (290:1.5cm){};	
				\node [blacknode] (A7) at (330:1.5cm){};
				\node [blacknode] (A8) at (10:1.5cm){};
				\node [blacknode] (A9) at (50:1.5cm){};  
				\node [blacknode] (u) at (0,0){}; 
				\draw(u)--(A1);
				\draw (A7)--(u)--(A4);
				\draw(A9) to [bend left=30] (A5);
				\draw(A2) to [bend right=30] (A6);
				\draw(A3) to [bend left=30] (A8);
				\node at (0.35,0.1) {$B_1$};
				\node at (-0.55,0.88) {$B_2$};
				\node at (-1.1,0) {$B_3$};
				\node at (1.8,0) {$G$};
				\node at (1.5,1) {$L$};
			\end{tikzpicture} 
		\end{minipage} 
		\caption{Two examples for Conjecture~\ref{main-conj}} 
	\end{figure}
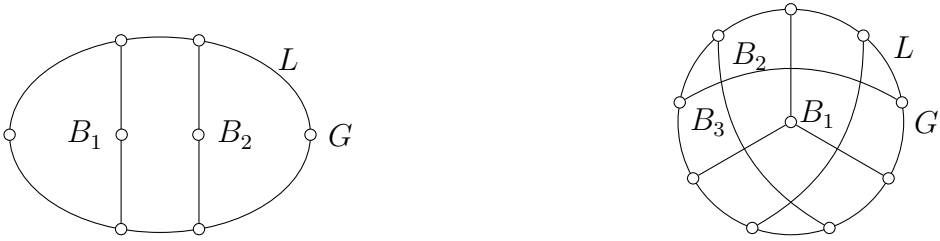\label{fig:2examples}

	The main result of this paper resolves Conjecture~\ref{main-conj} completely.
	
	\begin{theorem}\label{thm:main}
		Conjecture~\ref{main-conj} holds for all positive integers $k$.  
	\end{theorem}
	
	It is worth emphasizing that we prove this result by reducing the problem to one involving specified cycle coverings, which we then solve.  
	To state this result formally, we need to introduce one more concept: the \emph{symmetric difference} of two subgraphs $H_1$ and $H_2$ of a graph $G$, denoted $H_1 \triangle H_2$, which is the subgraph induced by the edge set $E(H_1) \triangle E(H_2)$.  
	The following is the corresponding result on cycle coverings. 
	
	\begin{theorem}
		\label{main-thm1}
		Let $L$ be a cycle in a $2$-connected graph $G$, and let $T_1, T_2, \ldots, T_k$ be $L$-bridges that are pairwise vertex-disjoint and induce a tree in $O_G(L)$, with each $T_i$ being a tree whose leaves are exactly the attachments of $T_i$. Then there exists a collection $\mathcal{C}$ of cycles in $G$ such that each edge of $L$ lies in exactly two cycles of $\mathcal{C}$, each edge not in $L$ lies in at least four cycles of $\mathcal{C}$, and for each cycle $C \in \mathcal{C}$, $L \triangle C$ is a cycle.
	\end{theorem} 
	
	Since the proof is short, we present this reduction immediately.
	
	\medskip
	
	\noindent{\bf Proof of Theorem~\ref{thm:main}, assuming Theorem~\ref{main-thm1}.}
	Let $B_1, B_2, \dots, B_k$ be $L$-bridges that satisfy the conditions of Theorem~\ref{thm:main}.
	For each $L$-bridge $B_i$, let $T_i$ be the largest tree in $B_i$ whose leaves are the attachments of $B_i$. 
	Define $H$ as the union of $L$ and $T_i$ for all $1\leq i \leq k$, which is 2-connected.
	By Theorem~\ref{main-thm1}, there exists a family $\mathcal{C}$ of cycles for $H$ as described. 
	Since $L$ is a longest cycle in $G$ and $L \triangle C$ is a cycle for every $C \in \mathcal{C}$, we have $|E(L \triangle C)| \leq |E(L)|$. Note that $|E(L \triangle C)| = |E(L) \setminus E(C)| + |E(C) \setminus E(L)|$ and $|E(L)| = |E(L) \setminus E(C)| + |E(L) \cap E(C)|$. Thus $|E(C) \setminus E(L)| \leq |E(L) \cap E(C)|$ for every $C \in \mathcal{C}$. Observe that $E(C) \setminus E(L) = \bigcup_{i=1}^{k} E(C) \cap E(T_i)$ and by the assumption, $B_1, B_2, \ldots, B_k$ are pairwise vertex-disjoint, so we have $\sum_{i=1}^{k} |E(C) \cap E(T_i)| = |E(C) \setminus  E(L)| \leq  |E(L) \cap E(C)|$ for every $C \in \mathcal{C}$.
	
	From the properties of the cycle family $\mathcal{C}$, each edge in $L$ lies in exactly two cycles of $\mathcal{C}$, and each edge in $\bigcup_{i=1}^{k} T_i$ lies in at least four. Therefore, we can derive that
	\begin{equation*} 
		4 \sum_{i=1}^{k} \lambda(B_i) = 4 \sum_{i=1}^{k} |E(T_i)| \leq \sum_{C \in \mathcal{C}} \sum_{i=1}^{k} |E(C) \cap E(T_i)| \leq \sum_{C \in \mathcal{C}} |E(C) \cap E(L)| = 2|E(L)|,
	\end{equation*}   
	which finishes the proof of Theorem~\ref{thm:main}. \qed
	
	\medskip
	
	The remainder of the paper is organized as follows. 
	In Section~\ref{sec-stronger thm}, we prove Theorem~\ref{thm:main} using 
	Theorem~\ref{main-thm1} and then present a slightly stronger version of 
	it that provides additional information on the cycle 
	family for inductive arguments. 
	Section~\ref{sec-pre} contains preliminary results needed for this stronger 
	version (Theorem~\ref{main-thm2}). 
	In Section~\ref{sec-proof}, we give the full proof of Theorem~\ref{main-thm2}. 
	Finally, we discuss some related open problems in the last section.

	\section{A strengthened version of Theorem~\ref{main-thm1}}\label{sec-stronger thm}
	
	Let $L$ be a cycle in a graph $G$. A natural idea for proving Theorem~\ref{main-thm1} is to use induction on the number of $L$-bridges. In the base case, where there are only one or two $L$-bridges, one can explicitly construct the corresponding family of cycles. On the other hand, for the case with at least three $L$-bridges, one may consider using induction hypothesis on $(G_1, L)$ and $(G_2, L)$, where $G_1$ and $G_2$ are subgraphs of $G$ with less $L$-bridges than $G$. This gives two families of cycles, $\mathcal{C}_1$ and $\mathcal{C}_2$, respectively. Then one can construct a new family $\mathcal{C}$ of cycles, satisfying the theorem, based on those in $\mathcal{C}_1 \cup \mathcal{C}_2$. However, to ensure that the induction works, both $O_L(G_1)$ and $O_L(G_2)$ must be trees. So we have to add some edges for certain cases in this process. Then, in the construction of $\mathcal{C}$, we also need  to address the cycles which contain the new edges. We will either delete these cycles or combine some pair of cycles by taking their symmetric difference. This necessitates a deeper understanding of the cycles involved, which motivates us to consider directed cycles instead. 
	
	We begin by defining four types of directed cycles (or, briefly, \emph{dicycles}), which play a central role in this paper. 
	Throughout, we denote by $\vv{H}$ a directed graph whose underlying graph is $H$. 
	For two vertices $x$ and $y$, we write $(x,y)$ (or $x \to y$) to denote an arc from $x$ to $y$, that is, an arc with tail $x$ and head $y$.

	Let $L$ be a cycle in a graph $G$, whose vertices are arranged in cyclic order (often assumed {\it clockwise} in a planar drawing of $L$).
	We denote by $v^+$ (resp. $v^-$) the next (resp. previous) vertex of $v$ on $L$ in this order.  Let $C$ be a cycle of $G$ that contains an edge $ab$, where $a$ is a vertex on $L$ and $b$ is not. We say that a dicycle $\vv{C}$ is of  \emph{type $ij$} with respect to $a$ if the following conditions are satisfied.
	\begin{itemize}
		\item $ij=00$: $(b,a)$ and $(a,a^+)$ are in $\vv{C}$; 
		\item $ij=01$: $(a,b)$ and $(a^+,a)$ are in $\vv{C}$;
		\item $ij=10$: $(b,a)$ and $(a,a^-)$ are in $\vv{C}$;
		\item $ij=11$: $(a,b)$ and $(a^-,a)$ are in $\vv{C}$.
	\end{itemize} 
	For convenience, we interpret $i=0$ (resp., $i=1$) as indicating that the dicycle contains $a^+$ (resp., $a^-$). Similarly, $j=0$  (resp., $j=1$) indicates that the dicycle contains the arc $(b,a)$ (resp., $(a,b)$). Note that for each attachment $v$ and $ij \in \mathbb{Z}_2^{2}$, it is possible that there are different dicycles of type $ij$ with respect to $v$.

	Given a graph $G$, let $\C(G)$ be the family of all possible dicycles of $G$. For a subfamily $\C \subseteq \C(G)$ and an edge $e \in E(G)$, let $$\mathcal{C}_e = \{\vv{C} \in \C: e\in E(C)\}.$$
	
	We are now ready to present the strengthened form of Theorem~\ref{main-thm1}, the proof of which will occupy the remainder of the paper.

	\begin{theorem}
		\label{main-thm2}
		Let $L$ be a cycle in a $2$-connected graph $G$ and  let $T_1, T_2, \ldots, T_s$ be all the $L$-bridges that are pairwise vertex-disjoint and induce a tree in $O_G(L)$, with each $T_i$ being a tree. Then there exists a subfamily $\mathcal{C} \subseteq \C(G)$ such that the following conditions hold.
		\begin{enumerate}[label= {(C\arabic*)}] 
			\item \label{C1} $|\mathcal{C}_e|=2$ for every $e \in E(L)$ and $|\mathcal{C}_e|\geq 4$ for every $e \in E(G)\setminus E(L)$. 
			\item\label{C2} For each dicycle in $\mathcal{C}$ with underlying graph $C$,  $L \triangle C$ is a cycle.
			\item \label{C3} Every dicycle in $\mathcal{C}$ contains either no or exactly two attachments of each $L$-bridge.
			\item \label{C4} For any $ij \in \mathbb{Z}_2^2$ and any vertex $x$ which is an attachment of some $L$-bridge, there is exactly one dicycle in $\mathcal{C}$ which is of type $ij$ with respect to $x$.
		\end{enumerate}   
	\end{theorem} 
	
	By considering the underlying graphs of the dicycles in $\C$, it is straightforward to see that Theorem~\ref{main-thm1} already follows from \ref{C1} and \ref{C2}.

	\section{Preliminary for the proof of Theorem~\ref{main-thm2}} 
	\label{sec-pre}
	
	In this section, we introduce several definitions and lemmas that will be useful later. In particular, we will present a method for constructing a special family of paths that covers a tree. For this purpose, we define an auxiliary digraph and introduce some of its properties in Section \ref{sec-digraph}. Then, in Section \ref{sec-path}, we use this auxiliary digraph to construct a family of paths in which each edge is covered at least four times.

	\subsection{An auxiliary digraph}\label{sec-digraph}
	
	Let $n$ and $k$ be two positive integers with $k \leq n$. We say a $k$-tuple $\eta = (p_1, p_2, \ldots, p_k)$ is a \emph{$k$-partition} of $n$ if $p_i \geq 1$ for each $1 \leq i \leq k$ and $\sum_{i=1}^{k} p_i = n$. Given a $k$-partition $\eta$ of $n$, we define a multidigraph $D_\eta$ with vertex-set $\bigcup _{i=1}^k\{v^i_1, v^i_2, \ldots, v^i_{p_i}\}$ and arcs as follows:
	\begin{itemize}
		\item $k=1$. We add the arcs $(v^1_i,v^1_{i+1})$ and  $(v^1_{i+1},v^1_i)$ for each $i=1,2,\ldots, n$, where we interpret $v^1_{n+1}$ as $v^1_1$.
		\item $k \geq 2$. We add (1) the arcs $(v^i_j, v^i_{j+1})$ and $(v^i_{j+1},v^i_j)$ for all $1 \leq i \leq k$ and $1 \leq j \leq p_i-1$ (if $p_i=1$, then this step is skipped); (2) the arcs $(v^i_{1}, v^{i+1}_{1})$ and  $(v^i_{p_i},v^{i+1}_{p_{i+1}})$ for all $1 \leq i \leq k$, where we interpret $v^{k+1}_1$ as $v^1_1$ and $v^{{}\,k+1}_{p_{k+1}}$ as $v^1_{p_1}$.
	\end{itemize} 
	
	\begin{figure}[h]
		\begin{minipage}[t]{0.23\textwidth}
			\centering
			\begin{tikzpicture}[>=latex,	
				blacknode/.style={circle, draw = black,  minimum size=0mm, inner sep=-0.1pt},
				] 
				\node [blacknode] (A1) at (36:1.6cm){$v^1_1$};
				\node [blacknode] (A2) at (72:1.6cm){$v^1_{\scalebox{0.5}{10}}$};
				\node [blacknode] (A3) at (108:1.6cm){$v^1_9$};
				\node [blacknode] (A4) at (144:1.6cm){$v^1_8$};
				\node [blacknode] (A5) at (180:1.6cm){$v^1_7$};
				\node [blacknode] (A6) at (216:1.6cm){$v^1_6$};	
				\node [blacknode] (A7) at (252:1.6cm){$v^1_5$};
				\node [blacknode] (A8) at (288:1.6cm){$v^1_4$};
				\node [blacknode] (A9) at (324:1.6cm){$v^1_3$};
				\node [blacknode] (A10) at (360:1.6cm){$v^1_2$};

				\draw [->](A1) [bend left =15] to (A2);
				\draw [->](A2) [bend left =15] to (A1);
				
				\draw [->](A3) [bend left =15] to (A2);
				\draw [->](A2) [bend left =15] to (A3);
				
				\draw [->](A3) [bend left =15] to (A4);
				\draw [->](A4) [bend left =15] to (A3);
				
				\draw [->](A4) [bend left =15] to (A5);
				\draw [->](A5) [bend left =15] to (A4);
				
				\draw [->](A5) [bend left =15] to (A6);
				\draw [->](A6) [bend left =15] to (A5);
				
				\draw [->](A6) [bend left =15] to (A7);
				\draw [->](A7) [bend left =15] to (A6);
				
				\draw [->](A7) [bend left =15] to (A8);
				\draw [->](A8) [bend left =15] to (A7);
				
				\draw [->](A9) [bend left =15] to (A8);
				\draw [->](A8) [bend left =15] to (A9);
				
				\draw [->](A9) [bend left =15] to (A10);
				\draw [->](A10) [bend left =15] to (A9);
				
				\draw [->](A1) [bend left =15] to (A10);
				\draw [->](A10) [bend left =15] to (A1);
			\end{tikzpicture} 
			\subcaption{$\eta=(10)$}
		\end{minipage}  
		\begin{minipage}[t]{0.25\textwidth}
			\centering
			\begin{tikzpicture}[>=latex,	
				blacknode/.style={circle, draw = black,  minimum size=0mm, inner sep=-0.1pt}] 
				\node [blacknode] (A1) at (36:1.6cm){$v^1_1$};
				\node [blacknode] (A2) at (72:1.6cm){$v^3_3$};
				\node [blacknode] (A3) at (108:1.6cm){$v^3_2$};
				\node [blacknode] (A4) at (144:1.6cm){$v^3_1$};
				\node [blacknode] (A5) at (180:1.6cm){$v^2_3$};
				\node [blacknode] (A6) at (216:1.6cm){$v^2_2$};	
				\node [blacknode] (A7) at (252:1.6cm){$v^2_1$};
				\node [blacknode] (A8) at (288:1.6cm){$v^1_4$};
				\node [blacknode] (A9) at (324:1.6cm){$v^1_3$};
				\node [blacknode] (A10) at (360:1.6cm){$v^1_2$};

				\draw [->](A1) to (A7);
				\draw [->](A2) to (A8);
				
				\draw [->](A3) [bend left =15] to (A2);
				\draw [->](A2) [bend left =15] to (A3);
				
				\draw [->](A3) [bend left =15] to (A4);
				\draw [->](A4) [bend left =15] to (A3);
				
				\draw [->](A4)  to (A1);
				\draw [->](A5)  to (A2);
				
				\draw [->](A5) [bend left =15] to (A6);
				\draw [->](A6) [bend left =15] to (A5);
				
				\draw [->](A6) [bend left =15] to (A7);
				\draw [->](A7) [bend left =15] to (A6);
				
				\draw [->](A7)  to (A4);
				\draw [->](A8)  to (A5);
				
				\draw [->](A9) [bend left =15] to (A8);
				\draw [->](A8) [bend left =15] to (A9);
				
				\draw [->](A9) [bend left =15] to (A10);
				\draw [->](A10) [bend left =15] to (A9);
				
				\draw [->](A1) [bend left =15] to (A10);
				\draw [->](A10) [bend left =15] to (A1);
			\end{tikzpicture} 
			\subcaption{$\eta=(4,3,3)$}
		\end{minipage}
		\begin{minipage}[t]{0.25\textwidth}
			\centering
			\begin{tikzpicture}[>=latex,	
				blacknode/.style={circle, draw = black, minimum size=0mm, inner sep=-0.1pt}] 
				\node [blacknode] (A1) at (36:1.6cm){$v^1_1$};
				\node [blacknode] (A2) at (72:1.6cm){$v^4_1$};
				\node [blacknode] (A3) at (108:1.6cm){$v^3_3$};
				\node [blacknode] (A4) at (144:1.6cm){$v^3_2$};
				\node [blacknode] (A5) at (180:1.6cm){$v^3_1$};
				\node [blacknode] (A6) at (216:1.6cm){$v^2_2$};	
				\node [blacknode] (A7) at (252:1.6cm){$v^2_1$};
				\node [blacknode] (A8) at (288:1.6cm){$v^1_4$};
				\node [blacknode] (A9) at (324:1.6cm){$v^1_3$};
				\node [blacknode] (A10) at (360:1.6cm){$v^1_2$};

				\draw [->](A1) to (A7);
				\draw [->](A2) to (A8);
				
				\draw [->](A3) [bend left =15] to (A4);
				\draw [->](A2) to (A1);
				
				\draw [->](A3) to (A2);
				\draw [->](A4) [bend left =15] to (A3);
				
				\draw [->](A4) [bend left =15] to (A5);
				\draw [->](A5)  to (A2);
				
				\draw [->](A5) [bend left =15] to (A4);
				\draw [->](A6) to (A3);
				
				\draw [->](A6) [bend left =15] to (A7);
				\draw [->](A7) [bend left =15] to (A6);
				
				\draw [->](A7) [bend right =15]  to (A5);
				\draw [->](A8) [bend right =15]  to (A6);
				
				\draw [->](A9) [bend left =15] to (A8);
				\draw [->](A8) [bend left =15] to (A9);
				
				\draw [->](A9) [bend left =15] to (A10);
				\draw [->](A10) [bend left =15] to (A9);
				
				\draw [->](A1) [bend left =15] to (A10);
				\draw [->](A10) [bend left =15] to (A1);
			\end{tikzpicture} 
			\subcaption{$\eta=(4,2,3,1)$}
		\end{minipage}
		\begin{minipage}[t]{0.23\textwidth}
			\centering
			\begin{tikzpicture}[>=latex,	
				blacknode/.style={circle, draw = black,  minimum size=0mm, inner sep=-0.1pt},
				] 
				\node [blacknode] (A1) at (36:1.6cm){$v^1_{\scalebox{0.5}{1}}$};
				\node [blacknode] (A2) at (72:1.6cm){$v^{\scalebox{0.5}{10}}_1$};
				\node [blacknode] (A3) at (108:1.6cm){$v^{_9}_1$};
				\node [blacknode] (A4) at (144:1.6cm){$v^{_8}_1$};
				\node [blacknode] (A5) at (180:1.6cm){$v^{_7}_1$};
				\node [blacknode] (A6) at (216:1.6cm){$v^{_6}_1$};	
				\node [blacknode] (A7) at (252:1.6cm){$v^{_5}_1$};
				\node [blacknode] (A8) at (288:1.6cm){$v^4_1$};
				\node [blacknode] (A9) at (324:1.6cm){$v^3_1$};
				\node [blacknode] (A10) at (360:1.6cm){$v^2_1$};

				\draw [->](A2) [bend right =15] to (A1);
				\draw [->](A2) [bend left =15] to (A1);
				
				\draw [->](A3) [bend left =15] to (A2);
				\draw [->](A3) [bend right =15] to (A2);
				
				\draw [->](A4) [bend right =15] to (A3);
				\draw [->](A4) [bend left =15] to (A3);
				
				\draw [->](A5) [bend right =15] to (A4);
				\draw [->](A5) [bend left =15] to (A4);
				
				\draw [->](A6) [bend right =15] to (A5);
				\draw [->](A6) [bend left =15] to (A5);
				
				\draw [->](A7) [bend right =15] to (A6);
				\draw [->](A7) [bend left =15] to (A6);
				
				\draw [->](A8) [bend right =15] to (A7);
				\draw [->](A8) [bend left =15] to (A7);
				
				\draw [->](A9) [bend left =15] to (A8);
				\draw [->](A9) [bend right =15] to (A8);
				
				\draw [->](A10) [bend right =15] to (A9);
				\draw [->](A10) [bend left =15] to (A9);
				
				\draw [->](A1) [bend left =15] to (A10);
				\draw [->](A1) [bend right =15] to (A10);
			\end{tikzpicture} 
			\subcaption{$\eta=(1,1,\ldots,1)$}
		\end{minipage}  
		\caption{Examples of $D_\eta$, where each $\eta$ is a partition of $10$.}
		\label{fig-digraph}
	\end{figure}
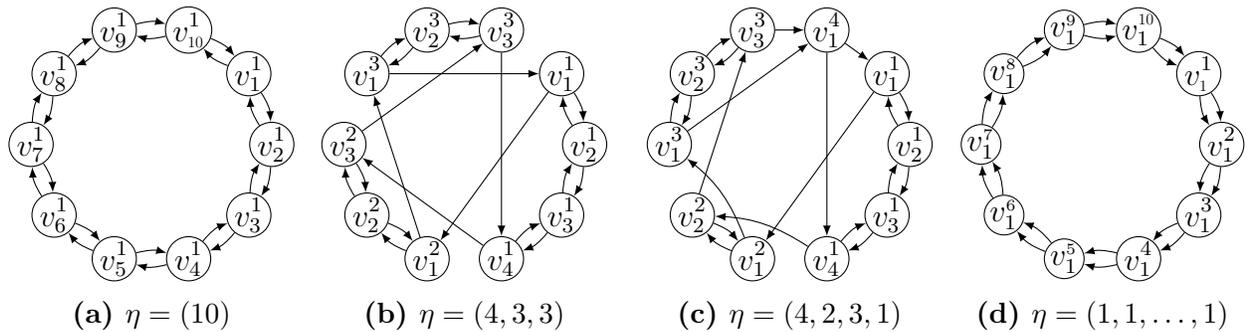 
	It is easy to check that  $D_\eta$ is a digraph in which every vertex has both in-degree $2$ and out-degree $2$. Note that if $p_i=p_{i+1}=1$ for some $i$, then $D_\eta$ contains multiple arcs. See Fig.~\ref{fig-digraph} for examples. 
	For convenience, we denote by $G_\eta$ the underlying multigraph of $D_\eta$, which may contain parallel edges. 
	
	\begin{lemma}
		\label{lem-twocycle}
		For each arc $(u,v)$ in $D_\eta$, there exist two dicycles containing $(u,v)$, and these two dicycles share only this arc. 
	\end{lemma}
	\begin{proof}
		Suppose $\eta$ is a $k$-partition. The statement obviously holds for $k=1$, one dicycle of length $2$ and the other of length $n$ suffice. Thus we assume that $k\geq 2$. It suffices to describe such two dicycles. 
		
		First assume that $u=v^i_j$ and $v=v^i_{j+1}$ for some $1 \leq i \leq k$ with $p_i \geq 2$ and $1 \leq j \leq p_i-1$. It is clear that $v^i_j \rightarrow v^i_{j+1} \rightarrow v^i_j$  and
		$$v^i_j \rightarrow v^i_{j+1} \rightarrow \cdots \rightarrow  v^i_{p_i} \rightarrow v^{i+1}_{p_{i+1}} \rightarrow \cdots \rightarrow  v^{i+1}_1 \rightarrow v^{i+2}_1 \rightarrow v^{i+3}_1 \rightarrow \cdots \rightarrow  v^{i}_1 \rightarrow v^{i}_2 \rightarrow \cdots \rightarrow  v^{i}_j$$ are the desired dicycles (the upper indices are taken modulo $k$).	
		
		Thus without loss of generality, we may assume that $u=v^i_{1}$ and $v=v^{i+1}_1$. Then  $v^i_1 \rightarrow v^{i+1}_{1} \rightarrow  v^{i+2}_{1} \rightarrow \cdots \rightarrow   v^{i}_1$ and 	$$v^i_1 \rightarrow v^{i+1}_{1} \rightarrow v^{i+1}_{2} \rightarrow \cdots \rightarrow  v^{i+1}_{p_{i+1}} \rightarrow v^{i+2}_{p_{i+2}} \rightarrow \cdots \rightarrow  v^{i}_{p_i} \rightarrow  v^{i}_{p_i-1} \rightarrow \cdots v^i_1$$ are the desired dicycles. We remark that the arcs $(v^{j}_1,v^{j+1}_1)$ and $(v^{j}_{p_j},v^{j+1}_{p_{j+1}})$ are essentially different if $p_j=p_{j+1}=1$, in which case they are multi-arcs sharing the same tail and head by the definitions of $D_\eta$.  
	\end{proof}  
	 
	Given a digraph $D(G)$ with underlying graph $G$, for $S,T \subsetneq V(G)$, we denote by $[S,T]_{D(G)}$ the set of arcs with tails in $S$ and heads in $T$, and $[S,T]_G$ the edges of $G$ joining $S$ and $T$. For a subset $\emptyset \neq X \subsetneq V(G)$, we denote by $\overline{X}$ the vertex-set  $V(G) \setminus X$.  
	
	\begin{lemma}
		\label{lem-atleast4}
		For any $k$-partition $\eta$ of $n$, $1 \leq k \leq n$ and any $\emptyset \neq X \subsetneq V(D_\eta)$, we have $|[X,\overline{X}]_{G_{\eta}}| \geq 4$.
	\end{lemma}
	\begin{proof}
		Assume for some  $\eta$ and a  subset $\emptyset \neq X \subsetneq V(D_\eta)$, we have $|[X,\overline{X}]_{G_\eta}| \leq 2$. Without loss of generality, assume that $[X,\overline{X}]_{D_{\eta}} \neq \emptyset$ and  $(u,v) \in [X,\overline{X}]_{D_{\eta}}$ is an arc, then for any dicycle in $D_\eta$ containing $(u,v)$, it must contain an arc in $[\overline{X},X]_{D_\eta}$. This contradicts Lemma \ref{lem-twocycle},  since we are unable to find two dicycles containing $(u,v)$ but only sharing $(u,v)$. Thus $|[X,\overline{X}]_{G_{\eta}}|  \geq 3$.  On the other hand, if $|[X,\overline{X}]_{G_\eta}|$ is odd, then the degree sum of the graph induced by $X$ is odd as each vertex in $G_\eta$ has even degree, but then it contradicts the Handshaking Lemma. Thus we have $|[X,\overline{X}]_{G_{\eta}}|  \geq 4$.
	\end{proof}
	
	\subsection{A family of dipaths of tree with labeled leaves}\label{sec-path}
	
	Given a tree $T$, let $\mathcal{P}(T)$ be the family of all possible dipaths of $T$ between leaves. For a dipath $\vv{P} \in \mathcal{P}(T)$ and an edge $e \in E(T)$, we say $\vv{P}$ covers $e$ if $e \in E(P)$. For $\mathcal{P} \subseteq \mathcal{P}(T)$ and an edge $e \in E(T)$, let $\mathcal{P}_e = \{\vv{P} \in \mathcal{P}: \text{$\vv{P}$ covers $e$}\}$.
	
	In this section, we will construct a subfamily $\mathcal{P} \subseteq \mathcal{P}(T)$, using the auxiliary digraph which we introduced in the beginning of Section~\ref{sec-digraph}, such that each edge of $T$ is covered by at least four dipaths in $\mathcal{P}$.
	
	For a tree $T$ and two distinct vertices $u,v \in T$, we denote by $T[u,v]$ the unique path in $T$ between $u$ and $v$, by $\vv{T}[u,v]$ the corresponding dipath from $u$ to $v$. We denote by $\partial(T)$ the set of leaves of $T$. 
	
	Recall that $D_{\eta}$ is the auxiliary digraph with respect to a $k$-partition $\eta = (p_1, p_2, \ldots, p_k)$ of some integer $n$. For a tree $T$ and a $k$-partition $\eta = (p_1, p_2, \ldots, p_k)$ of $|\partial(T)|$, an \emph{$\eta$-labeling} $\ell$ of $\partial(T)$ is a bijection from $\partial(T)$ to $V(D_{\eta})$.

	Given a tree $T$, a $k$-partition $\eta=(p_1, p_2, \ldots, p_k)$ of $|\partial(T)|$ and an $\eta$-labeling $\ell$ of $\partial(T)$, we define $\mathcal{P}(T,\eta, \ell)$ to be the collection of all dipaths $\vv{T}[u,v]$, where $u$ and $v$ are leaves of $T$ such that $(\ell(u),\ell(v))$ is an arc in the auxiliary digraph $D_{\eta}$.

	The following lemma analyzes the number of dipaths that cover an edge in   $\mathcal{P}(T,\eta,\ell)$.  
	
	\begin{lemma}
		\label{lem-coverpath}
		Let $T$ be a tree,  $\eta = (p_1, p_2, \ldots, p_k)$ be a $k$-partition of $|\partial(T)|$ and $\ell$ be a $\eta$-labeling of $\partial(T)$. Then for any edge $e \in E(T)$,  $|\mathcal{P}_e(T,\eta,\ell)|\geq 4$. 
	\end{lemma}
	\begin{proof}
		Assume $e \in E(T)$, $T_1$ and $T_2$ are the two components of $T-e$. Let $X_1=\{\ell(v)|v \in \partial(T_1)\}$  and $X_2=\{\ell(v)|v \in \partial(T_2)\}$. Since $\ell$ is a one-to-one correspondence from $\partial(T)$ to $V(D_{\eta})$, and $\partial(T_1) \cap \partial(T_2) =\emptyset$, $\partial(T)= \partial(T_1) \cup \partial(T_2)$, it follows that $V(D_{\eta}) = X_1 \cup X_2$ and $X_1 \cap X_2 = \emptyset$, that is $X_2=\overline{X}_1$ in $D_{\eta}$. By Lemma~\ref{lem-atleast4},  $|[X_1,X_2]_{G_{\eta}}| \geq 4$. This implies that there exist at least four distinct ordered vertex pairs $(a,b) \in \partial(T_1) \times \partial(T_2)$ such that $(\ell(a),\ell(b)) \in [X_1,X_2]_{D_\eta} \cup [X_2,X_1]_{D_\eta}$. For each such pair $(a,b)$, either $\vv{T}[a,b] \in \mathcal{P}(T,\eta, \ell)$ or $\vv{T}[b,a] \in \mathcal{P}(T, \eta, \ell)$, and both $\vv{T}[a,b]$ and $\vv{T}[b,a]$ cover the edge $e$. Thus, $|\mathcal{P}_e(T, \eta, \ell)|\geq 4$.
	\end{proof}
	
	\section{Proof of Theorem \ref{main-thm2}} 
	\label{sec-proof} 
	
	For convenience, we say $\mathcal{C}$ is \emph{feasible} for $(G,L)$ if $\C \subseteq \C(G)$ and $\C$ satisfies \ref{C1}-\ref{C4}. For two vertices $u$ and $v$ on $L$, let $L[u,v]$ be the segment of $L$ from $u$ to $v$ in clockwise direction. We denote by $\vv{L}[u,v]$ the copy of $L[u,v]$ directed from $u$ to $v$, and $\cevv{L}[u,v]$ the copy of $L[u,v]$ directed from $v$ to $u$. For a subgraph $H$ of $G$, let $\vv{C}|_H$ be the subdigraph (not necessarily connected) of $\vv{C}$ whose underlying graph is induced by the edge set $E(C) \cap E(H)$, and let 
	$\mathcal{C}|_H = \{\vv{C}|_H: \vv{C} \in \C \}.$ 
	
	\medskip
	\noindent
	{\bf Basic step: $s=1$} 
	\medskip
	
	Assume that $v_1, v_2, \ldots, v_n$ are leaves of $T_1$, arranged in clockwise cyclic order on $L$.
	We construct $\C \subseteq \C(G)$ as follows: for each $1 \leq i \leq n$, we add two dicycles $$\vv{T}[v_i,v_{i+1}] \cup \cevv{L}[v_i,v_{i+1}] \ \text{and} \  \vv{T}[v_{i+1}, v_i] \cup \vv{L}[v_i,v_{i+1}]$$ to $\mathcal{C}$, where the indices are taken modulo $n$. We shall check that $\C$ satisfies \ref{C1}-\ref{C4}. Note that \ref{C2}-\ref{C4} are straightforward, we omit the verification and only focus on \ref{C1}. 
	
	It is clear that there is a one-to-one correspondence between $\mathcal{C}|_{T}$ and $\mathcal{P}(T,\eta,\ell)$, where $\eta=(n)$ is a $1$-partition of $n$, and $\ell: \partial(T) \to \{v_1,v_2,\ldots,v_n\}$ is the $\eta$-labeling of $\partial(T)$. Therefore, by Lemma \ref{lem-coverpath}, $|\mathcal{C}_e| = |\mathcal{P}_e(T,\eta,\ell)| \geq 4$ for every $e \in E(G) \setminus E(L)$. On the other hand, observe that by our construction, $|\mathcal{C}_e| = 2$ for each $e \in E(L)$. Thus $\mathcal{C}$ satisfies \ref{C1}.  
	
	\medskip
	\noindent
	{\bf Basic step: $s=2$} 
	\medskip 
	
	In this case, $G$ contains exactly two $L$-bridges $T_1$ and $T_2$. Let 
	$$u^1_1, \ldots, u^1_{p_1}, v^1_1, \ldots, v^1_{q_1}, u^2_1, \ldots, u^2_{p_2}, v^2_1, \ldots, v^2_{q_2}, \ldots, u^k_1,\ldots,u^k_{p_k},v^k_1,\ldots,v^k_{q_k}$$
	be all the attachments listed in clockwise cyclic order on $L$, where $p_i, q_i \geq 1$ for each $1 \leq i \leq k$, the vertices $u^i_j$'s are the attachments (leaves) of $T_1$, and the vertices $v^i_j$'s are the attachments (leaves) of $T_2$. Note that $k \geq 2$ because the overlap graph $O_{G}(L)$ is connected.  Thus $\eta_1=(p_1, p_2, \ldots, p_k)$ is a $k$-partition of $|\partial(T_1)|$, and $\eta_2=(q_1,q_2,\ldots,q_k)$ is a $k$-partition of $|\partial(T_2)|$. Let $\ell_i$ be the corresponding $\eta_i$-labeling of $\partial(T_i)$ for $i=1,2$.   
	
	We now construct $\C$ as follows (See Fig.~\ref{fig:s=2} for illustration). For  $1 \leq i \leq k$,  we add the following dicycles to $\C$, where the (superscript) index $0$ is interpreted as $k$, and $k+1$  as $1$, 
	\begin{enumerate}[label= {(\arabic*)}] 
		\item \label{s21} $\vv{T_1}[u^i_j,u^i_{j+1}] \cup \cevv{L}[u^i_{j},u^i_{j+1}]$ and $\vv{T_1}[u^i_{j+1}, u^i_j] \cup \vv{L}[u^i_j,u^i_{j+1}]$ for each $1 \leq j \leq p_i-1$ if $p_i \geq 2$.
		\item \label{s22} $\vv{T_2}[v^i_j,v^i_{j+1}] \cup \cevv{L}[v^i_{j},v^i_{j+1}]$ and $\vv{T_2}[v^i_{j+1}, v^i_j] \cup \vv{L}[v^i_j,v^i_{j+1}]$ for each $1 \leq j \leq q_i-1$ if $q_i \geq 2$.
		\item \label{s23} $\vv{T_2}[v^i_1, v^{i+1}_1] \cup \cevv{L}[u^{i+1}_{p_{i+1}},v^{i+1}_1] \cup \vv{T_1}[ u^{i+1}_{p_{i+1}}, u^i_{p_i}] \cup \vv{L}[u^i_{p_i},v^i_1]$. 
		\item \label{s24} $\vv{T_2}[v^{i-1}_{q_{i-1}}, v^i_{q_i}] \cup \vv{L}[v^i_{q_i},u^{i+1}_1] \cup \vv{T_1}[u^{i+1}_1, u^i_1] \cup \cevv{L}[v^{i-1}_{q_{i-1}},u^i_1]$.
	\end{enumerate} 
	
	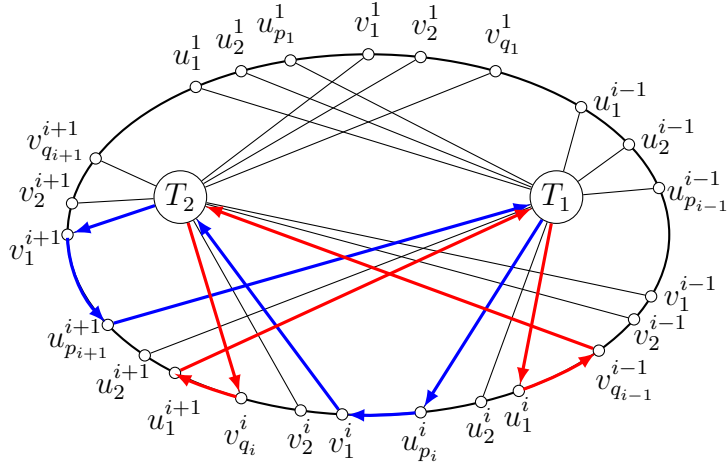
\begin{figure}[h] 
		\centering
		\begin{tikzpicture}[>=latex,
			Tnode/.style={circle, draw= black, fill = white, minimum size=7mm, inner sep=0pt},
			snode/.style={circle, draw=black,fill=white, minimum size=1.5mm, inner sep=0pt}] 
			 
			\draw[thick] (0,0) ellipse (4cm and 2.4cm);
			 
			\node [Tnode] (T2) at (-2.5,0.5){}; 
			\node [Tnode] (T1) at (2.5,0.5){};  
			\node at (-2.5,0.5){$T_2$};
			\node at (2.5,0.5){$T_1$};
			 
			\node[snode, label={[xshift=-3pt] $u^1_1$}] (u11) at (125:4cm and 2.4cm) {};
			\node[snode, label={[xshift=-4pt] $u^1_2$}] (u12) at (115:4cm and 2.4cm) {};
			\node[snode, label={[xshift=-5pt] $u^1_{p_1}$}] (u1p) at (105:4cm and 2.4cm) {};
			
			\node[snode, label={[xshift=0pt] $v^1_1$}] (v11) at (90:4cm and 2.4cm) {};
			\node[snode, label={[xshift=2pt] $v^1_2$}] (v12) at (80:4cm and 2.4cm) {};
			\node[snode, label={[xshift=3pt] $v^1_{q_1}$}] (v1q) at (65:4cm and 2.4cm) {};
			
			\node[snode, label={[xshift=15pt,yshift=-10pt] $u^{i-1}_1$}] (ux1) at (45:4cm and 2.4cm) {};
			\node[snode, label={[xshift=15pt,yshift=-10pt] $u^{i-1}_2$}] (ux2) at (30:4cm and 2.4cm) {};
			\node[snode, label={[xshift=15pt,yshift=-16pt] $u^{i-1}_{p_{i-1}}$}] (uxp) at (15:4cm and 2.4cm) {};
			
			\node[snode, label={[xshift=15pt, yshift=-12pt] $v^{i-1}_1$}] (vx1) at (-20:4cm and 2.4cm) {};
			\node[snode, label={[xshift=10pt, yshift=-18pt] $v^{i-1}_2$}] (vx2) at (-28:4cm and 2.4cm) {};
			\node[snode, label={[xshift=10pt, yshift=-26pt] $v^{i-1}_{q_{i-1}}$}] (vxq) at (-40:4cm and 2.4cm) {};
			
			\node[snode, label={[yshift=-22pt] $u^i_1$}] (uy1) at (-60:4cm and 2.4cm) {};
			\node[snode, label={[yshift=-22pt] $u^i_2$}] (uy2) at (-68:4cm and 2.4cm) {};
			\node[snode, label={[yshift=-25pt] $u^i_{p_i}$}] (uyp) at (-80:4cm and 2.4cm) {};
			
			\node[snode, label={[yshift=-23pt] $v^i_1$}] (vy1) at (-95:4cm and 2.4cm) {};
			\node[snode, label={[yshift=-23pt] $v^i_2$}] (vy2) at (-103:4cm and 2.4cm) {};
			\node[snode, label={[yshift=-28pt] $v^i_{q_i}$}] (vyq) at (-115:4cm and 2.4cm) {};
			
			\node[snode, label={[yshift=-30pt] $u^{i+1}_1$}] (uz1) at (-130:4cm and 2.4cm) {};
			\node[snode, label={[xshift=-8pt, yshift=-23pt] $u^{i+1}_2$}] (uz2) at (-138:4cm and 2.4cm) {};
			\node[snode, label={[xshift=-11pt, yshift=-20pt] $u^{i+1}_{p_{i+1}}$}] (uzp) at (-150:4cm and 2.4cm) {};
			
			\node[snode, label={[xshift=-12pt, yshift=-18pt] $v^{i+1}_1$}] (vz1) at (180:4cm and 2.4cm) {};
			\node[snode, label={[xshift=-10pt, yshift=-8pt] $v^{i+1}_2$}] (vz2) at (170:4cm and 2.4cm) {};
			\node[snode, label={[xshift=-15pt, yshift=-8pt] $v^{i+1}_{q_{i+1}}$}] (vzq) at (155:4cm and 2.4cm) {};
			 
			\foreach \id in {u11, u12, u1p, ux1, ux2, uxp, uy1, uy2, uyp, uz1, uz2, uzp} {
				\draw (T1) -- (\id);
			}
			 
			\foreach \id in {v11, v12, v1q, vx1, vx2, vxq, vy1, vy2, vyq, vz1, vz2, vzq} {
				\draw (T2) -- (\id);
			} 
			
			\draw [blue, line width=1.2pt,->] (T1)-- (uyp);
			\draw [blue, line width=1.2pt,->] (uzp)--(T1); 
			\draw [blue, line width=1.2pt,->] (-179: 4cm and 2.4cm) arc (-179: -151: 4cm and 2.4cm);
			\draw [blue, line width=1.2pt,->] (T2)-- (vz1);
			\draw [blue, line width=1.2pt,->] (vy1)--(T2); 
			\draw [blue, line width=1.2pt,->] (-81: 4cm and 2.4cm) arc (-81:-94: 4cm and 2.4cm);
			
			\draw [red, line width=1.2pt,->] (T1)--(uy1);
			\draw [red, line width=1.2pt,->] (uz1)--(T1); 
			\draw [red, line width=1.2pt,->] (-116: 4cm and 2.4cm) arc (-115:-129: 4cm and 2.4cm);
			\draw [red, line width=1.2pt,->] (T2)--(vyq);
			\draw [red, line width=1.2pt,->] (vxq)--(T2); 
			\draw [red, line width=1.2pt,->] (-59: 4cm and 2.4cm) arc (-59:-41:4cm and 2.4cm);
		\end{tikzpicture} 
		\caption{Illustration for the construction of $\C$. The blue dicycle corresponds to \ref{s23} and the red one corresponds to \ref{s24}.}
		\label{fig:s=2} 
	\end{figure} 
	
	By the construction of $\C$, it is easy to check that $\C$ satisfies  \ref{C3}-\ref{C4}, we omit the  details and only focus on the  verification of \ref{C1} and  \ref{C2}.
	
	Observe that for $i=1,2$, the underlying paths in $\mathcal{C}|_{T_i}$ and $\mathcal{P}(T_i,\eta_i,\ell_i)$ are identical, so there is a natural one-to-one correspondence between $\mathcal{C}|_{T_i}$ and  $\mathcal{P}(T_i,\eta_i,\ell_i)$.  Thus by Lemma \ref{lem-coverpath}, for each edge $e \in E(T_i)$, $|\mathcal{C}_e| = |\mathcal{P}_e(T_i, \eta_i, \ell_i)|  \geq 4$. On the other hand, $|\mathcal{C}_e|=2$ for each edge $e \in E(L)$. Therefore $\mathcal{C}$ satisfies \ref{C1}.
	
	The dicycles constructed in \ref{s21} and \ref{s22} satisfy \ref{C2} obviously, we focus on the rest two classes. For $\vv{C} = \vv{T_2}[v^i_1, v^{i+1}_1] \cup \cevv{L}[u^{i+1}_{p_{i+1}},v^{i+1}_1] \cup \vv{T_1}[u^{i+1}_{p_{i+1}}, u^i_{p_i}] \cup \vv{L}[u^i_{p_i},v^i_1]$, we have that $$C \triangle L = L[v^{i+1}_1, u^i_{p_i}] \cup T_1[u^i_{p_i}, u^{i+1}_{p_{i+1}}] \cup L[v^i_1, u^{i+1}_{p_{i+1}}] \cup T_2[v^i_1, v^{i+1}_1]$$ is a cycle. For $\vv{C'} = \vv{T_2}[v^{i-1}_{q_{i-1}}, v^i_{q_i}] \cup \vv{L}[v^i_{q_i},u^{i+1}_1] \cup \vv{T_1}[u^{i+1}_1, u^i_1] \cup \cevv{L}[v^{i-1}_{q_{i-1}},u^i_1]$, 
	$$C' \triangle L = L[u^{i+1}_1, v^{i-1}_{q_{i-1}}] \cup T_2[v^{i-1}_{q_{i-1}}, v^i_{q_i}] \cup L[u^i_1, v^i_{q_i}] \cup T_1[u^i_1, u^{i+1}_1]$$ is also a cycle. Thus $\mathcal{C}$ satisfies \ref{C2}. 
	
	\medskip
	\noindent
	{\bf Induction step: $s \geq 3$}
	
	Recall that $T_1,T_2, \ldots, T_s$ are all the $L$-bridges. We first claim the following.
	
	\begin{claim}
		\label{claim-uv}
		There exist two vertices $u$ and $v$ on $L$ such that the segment $L[u,v]$ contains all the attachments of some $T_i$, which is a leaf in $O_G(L)$, and contains no attachments of any other $L$-bridge that does not overlap with $T_i$.
	\end{claim}  
	\begin{proof} 
		We choose two vertices $u,v$ such that $L[u,v]$ is as short as possible and contains all the attachments of some $T_i$ which is a leaf in $O_G(L)$. By this choice, both $u$ and $v$ are attachments of $T_i$. Note that for any other $T_j$ which does not overlap with $T_i$, if $L[u,v]$ contains one attachment of $T_j$, then it contains all the attachments of $T_j$, otherwise $T_i$ and $T_j$ would overlap.
		Suppose the claim is not true. Then $L[u,v]$ must contain all the attachments of $T_j$ for some $j \neq i$, where $T_j$ does not overlap with $T_i$. It is clear that $\partial(T_j)$ lies in $L[u^+,v^-]$ as the $L$-bridges are pairwise disjoint. We claim that $L[u^+,v^-]$ contains all the attachments of some $L$-bridge other than $T_i$ that is a leaf in $O_G(L)$, which contradicts the choice of $u$ and $v$. Indeed, if $T_j$ is a leaf in $O_G(L)$, then we are done. Assume $T_j$ is not a leaf. Since $O_G(L)$ is a tree, there exists a new leaf $T_k$ such that the unique path $P_{jk}$ in $O_G(L)$ between $T_j$ and $T_k$ does not contain the vertex $T_i$. Then by an easy inductive argument, it is easy to see that any $T_\ell \in V(P_{jk})$ (note that $T_\ell$ is not adjacent to $T_i$, i.e., $T_\ell$ and $T_i$ are not overlap) has an attachment in $L[x,y]$. Thus all attachments of $T_k$ are contained in $L[u^+,v^-]$.
	\end{proof} 
	
	Without loss of generality, by Claim~\ref{claim-uv}, we may assume that $u$ and $v$ are two vertices on $L$ such that $L[u,v]$ contains all the attachments of $T_1$ and $L[u,v]$ is as short as possible, and $T_1$ is a leaf in the overlap graph $O_G(L)$ that overlaps only with $T_{2}$. Assume 
	$$u^1_1, \ldots, u^1_{p_1}, v^1_1, \ldots, v^1_{q_1}, u^2_1, \ldots, u^2_{p_2}, v^2_1, \ldots, v^2_{q_2}, \ldots, v^{k-1}_1, \ldots, v^{k-1}_{q_{k-1}}, u^k_1, \ldots, u^k_{p_k},$$ are the attachments listed in clockwise cyclic order  on $L[u,v]$, where $k \geq 2$, $p_i, q_i \geq 1$,  the vertices $u^i_j$'s are attachments of $T_1$, and the vertices $v^i_j$'s are (partial) attachments of $T_2$.

	For two attachments $x, y$ on $L$, we say that $x$ and $y$ \emph{witness} each other if there are no other attachments on the segments $L[x, y]$ or $L[y, x]$. 
	By our choice of $u$ and $v$, we have $u = u^1_1$ and $v = u^k_{p_k}$; 
	hence, only $u^1_1$ or $u^k_{p_k}$ can possibly witness attachments of $T_i$ for $3 \le i \le s$.
	
	The remainder of the proof is divided into three subsections, depending on 
	how many vertices in $\{u^1_1, u^k_{p_k}\}$ can witness attachments of $T_i$ 
	for some $3 \le i \le s$.

	\subsection{Neither $u^1_1$ nor $u^k_{p_k}$ witness attachments of $T_i$ for $3 \le i \le s$}
	\label{sec-case 1} 
	
	Starting from $u^1_1$ and moving in the counterclockwise direction, let $a$ be the first attachment of $T_2$ that can be witnessed by $u^1_1$.  
	Starting from $u^k_{p_k}$ and moving in the clockwise direction, let $b$ be the first attachment of $T_2$ that may witness an attachment of $T_i$ for $3 \le i \le s$. See Fig.~\ref{fig-case1} for an illustration. Let $T_{21}$ be the subtree of $T_2$ whose leaves are precisely the attachments on $L[b,a]$, and let $T_{22}$ be the subtree of $T_2$ whose leaves are precisely the attachments on $L[a,b]$. We have the following.
	
	\begin{claimA}
		\label{claim-union} 
		$E(T_2) = E(T_{21}) \cup E(T_{22})$. 
	\end{claimA}
	\begin{proof}
		By the definition of $T_{21}$ and $T_{22}$,  $\partial(T_2) = \partial(T_{21}) \cup \partial(T_{22})$ and $\partial(T_{21}) \cap \partial(T_{22}) \neq \emptyset$. Since $E(T_{21}) \cup E(T_{22}) \subseteq E(T_2)$, it suffices to show that $E(T_2) \subseteq E(T_{21}) \cup E(T_{22})$. Suppose to the contrary that there exists an edge $e \in E(T_2)$, but $e \notin E(T_{21}) \cup E(T_{22})$. Then either $T_{21}$ and $T_{22}$ are contained in the two components of $T_2-e$, respectively, or $T_{21} \cup T_{22}$ is contained in one component of $T_2-e$. The former case contradicts the condition that $\partial(T_{21}) \cap \partial(T_{22}) \neq \emptyset$, and the latter case contradicts the condition that $\partial(T_2) = \partial(T_{21}) \cup \partial(T_{22})$.
	\end{proof}

	\begin{figure}[h]
		\centering 
		\begin{tikzpicture}[>=latex,
			Tnode/.style={circle, draw= none, fill= magenta!50,opacity=0.3, minimum size=1.3cm, inner sep=0pt}, 
			bnode/.style={circle, draw=black,fill=green, minimum size=0.8mm, inner sep=0pt},
			snode/.style={circle, draw=black,fill=red, minimum size=0.8mm, inner sep=0pt},
			rnode/.style={circle, draw=black,fill=orange, minimum size=0.8mm, inner sep=0pt},
			Bnode/.style={circle, draw=black,fill=white, minimum size=1.5mm, inner sep=0pt}] 
			
			\draw ($(0,0)+(0:4.5 and 2.5)$) arc (0:360:4.5 and 2.5); 
			
			\node [Bnode] (u11) at ($(0,0)+(130:4.5 and 2.5)$){}; 
			\node [Bnode] (u12) at ($(0,0)+(135:4.5 and 2.5)$){};
			\node [Bnode] (u13) at ($(0,0)+(140:4.5 and 2.5)$){};
			
			\node [Bnode] (u21) at ($(0,0)+(80:4.5 and 2.5)$){};
			\node [Bnode] (u22) at ($(0,0)+(85:4.5 and 2.5)$){};
			\node [Bnode] (u23) at ($(0,0)+(90:4.5 and 2.5)$){};
			\node [Bnode] (u24) at ($(0,0)+(95:4.5 and 2.5)$){};
			
			\node [Bnode] (u31) at ($(0,0)+(30:4.5 and 2.5)$){};
			\node [Bnode] (u32) at ($(0,0)+(35:4.5 and 2.5)$){};
			\node [Bnode] (u33) at ($(0,0)+(40:4.5 and 2.5)$){}; 
			
			\node [Bnode] (u41) at ($(0,0)+(-30:4.5 and 2.5)$){};
			\node [Bnode] (u42) at ($(0,0)+(-35:4.5 and 2.5)$){};
			\node [Bnode] (u43) at ($(0,0)+(-40:4.5 and 2.5)$){}; 
			
			\node [Bnode] (u51) at ($(0,0)+(-77:4.5 and 2.5)$){};
			\node [Bnode] (u52) at ($(0,0)+(-82:4.5 and 2.5)$){};
			\node [Bnode] (u53) at ($(0,0)+(-87:4.5 and 2.5)$){}; 
			
			\node [Bnode] (u61) at ($(0,0)+(-130:4.5 and 2.5)$){};
			\node [Bnode] (u62) at ($(0,0)+(-135:4.5 and 2.5)$){};
			\node [Bnode] (u63) at ($(0,0)+(-140:4.5 and 2.5)$){}; 
			
			\node [Bnode] (v11) at ($(0,0)+(55:4.5 and 2.5)$){};
			\node [Bnode] (v12) at ($(0,0)+(60:4.5 and 2.5)$){};
			\node [Bnode] (v13) at ($(0,0)+(65:4.5 and 2.5)$){};

			\node [Bnode] (v21) at ($(0,0)+(0:4.5 and 2.5)$){};
			\node [Bnode] (v22) at ($(0,0)+(5:4.5 and 2.5)$){};
			\node [Bnode] (v23) at ($(0,0)+(10:4.5 and 2.5)$){};
			
			\node [Bnode] (v31) at ($(0,0)+(-55:4.5 and 2.5)$){};
			\node [Bnode] (v32) at ($(0,0)+(-65:4.5 and 2.5)$){}; 
			\node [Bnode] (v33) at ($(0,0)+(-60:4.5 and 2.5)$){}; 
			
			\node [Bnode] (w11) at ($(0,0)+(105:4.5 and 2.5)$){};
			\node [Bnode] (w12) at ($(0,0)+(110:4.5 and 2.5)$){};
			\node [Bnode] (w13) at ($(0,0)+(115:4.5 and 2.5)$){};

			\node [Bnode] (w21) at ($(0,0)+(175:4.5 and 2.5)$){};
			\node [Bnode] (w22) at ($(0,0)+(180:4.5 and 2.5)$){};
			\node [Bnode] (w23) at ($(0,0)+(185:4.5 and 2.5)$){};
			
			\node [Bnode] (w31) at ($(0,0)+(-105:4.5 and 2.5)$){};
			\node [Bnode] (w32) at ($(0,0)+(-110:4.5 and 2.5)$){};
			\node [Bnode] (w33) at ($(0,0)+(-115:4.5 and 2.5)$){};

			\node [snode] (T1) at (0.5,0){};
			\node at (0,-0.75){$T_2$};
			
			\draw (T1) -- (u11);
			\draw (0.3,0) -- (u12);
			\draw (0,0) -- (u13);
			
			\draw (0.7,0) -- (u21);
			\draw (0.6,0) -- (u22);
			\draw (0.55,0) -- (u23);
			\draw (T1) -- (u24);
			
			\draw (0.5,-0.1) -- (u31);
			\draw (T1) -- (u32);
			\draw (0.5,0.1) -- (u33);

			\draw (T1) -- (u41);
			\draw (T1) -- (u42);
			\draw (T1) -- (u43); 
			
			\draw (T1) -- (u51);
			\draw (T1) -- (u52);
			\draw (T1) -- (u53); 
			
			\draw (T1) -- (u61);
			\draw (T1) -- (u62);
			\draw (T1) -- (u63); 
			 
			\coordinate   (Am) at ($(0.5,0)+(0:0.65)$);
			\coordinate   (Bm) at (0.5,-0.7);
			\coordinate   (Cm) at (0.3,-0.6);
			\coordinate   (Dm) at (0.1,-0.5);
			\coordinate   (Em) at (-0.1,-0.4);
			\coordinate   (Fm) at (-0.2,-0.1);
			\coordinate   (Gm) at (0.3,0.1);
			\coordinate   (Hm) at ($(0.5,0)+(90:0.6)$);
			
			\draw[draw=none, fill=brown!30] (Am) to [closed, curve through = {(Bm) (Cm) (Dm) (Em) (Fm) (Gm) (Hm)}] (Am);

			\node at (2.9,-0.5){$T_1$};  
			
			\node[snode]  (A0) at (0.5,0){};	 
			\node[snode]  (A1) at ($(0.5,0)+(45:0.15)$) {};  
			\node[snode]  (A3) at ($(0.5,0)+(210:0.15)$) {};
			\draw (A0)--(A1); 
			\draw (A0)--(A3); 
			\node[snode]  (A4) at ($(0.5,0)+(-30:0.15)$) {};
			\draw (A0)--(A4); 
			\foreach \t in {1,6,9}{
				\node[snode]  (B\t) at ($(0.5,0)+(30*\t+30:0.33)$) {};} 
			\node[snode]  (B10) at ($(0.5,0)+(340:0.33)$){}; 
			\draw (A1)--(B1);   
			\draw (A3)--(B6);  
			\draw (A4)--(B9);
			\draw (A4)--(B10); 
			\foreach \t in {1,2,6,7,...,12}{
				\node[snode]  (C\t) at ($(0.5,0)+(30*\t+15:0.56)$) {};}
			\draw (B1)--(C1); 	
			\draw (B1)--(C2); 
			\draw (B6)--(C6); 	
			\draw (B6)--(C7); 
			\draw (A0)--(C8); 
			\draw (B9)--(C9); 	
			\draw (B9)--(C10);
			\draw (B10)--(C11);
			\draw (A4)--(C12);

			\node [snode] (T2) at (2.5,0){}; 
			\draw (T2) -- (v11);
			\draw (T2) -- (v12);
			\draw (T2) -- (v13);
			
			\draw (T2) -- (v21);
			\draw (T2) -- (v22);
			\draw (T2) -- (v23);
			\draw (T2) -- (v31);
			\draw (T2) -- (v32);
			\draw (T2) -- (v33);

			\coordinate   (Ar) at (2.6,0.52);
			\coordinate   (Br) at (2.2,0.25);
			\coordinate   (Cr) at (2,-0.2);
			\coordinate   (Dr) at (2.1,-0.4);
			\coordinate   (Er) at (2.3,-0.3);
			\coordinate   (Fr) at (2.5,-0.15);
			\coordinate   (Gr) at (2.8,-0.1);
			\coordinate   (Hr) at (3,0);
			
			\draw[draw=none, fill=cyan!30] (Ar) to [closed, curve through = {(Br) (Cr) (Dr) (Er) (Fr) (Gr) (Hr)}] (Ar);
			\node[rnode]  (D0) at (2.5,0){};	 
			
			\node[rnode]  (D1) at ($(2.5,0)+(60:0.13)$) {};
			\draw (D0)--(D1);
			
			\node[rnode]  (D2) at ($(2.5,0)+(135:0.13)$) {};
			\draw (D0)--(D2);
			
			\node[rnode]  (D3) at ($(2.5,0)+(210:0.13)$) {};
			\draw (D0)--(D3);

			\foreach \t in {1,6}{
				\node[rnode]  (E\t) at ($(2.5,0)+(30*\t+30:0.27)$) {};}
			
			\node[rnode]  (E3) at ($(2.5,0)+(130:0.27)$) {}; 
			
			\draw (D1)--(E1); 	
			\draw (D2)--(E3);
			
			\draw (D3)--(E6);  
			
			\foreach \t in {1,2,6,7,12}{
				\node[rnode]  (F\t) at ($(2.5,0)+(30*\t+15:0.46)$) {};}
			\draw (E1)--(F1); 	
			\draw (E1)--(F2);  
			\draw (E6)--(F6); 	
			\draw (E6)--(F7); 
			\draw (D0)--(F12);

			\node [snode] (T0) at (-2,0){};	
			\draw (T0) -- (w11);
			\draw (T0) -- (w12);
			\draw (T0) -- (w13);
			\draw (T0) -- (w21);
			\draw (T0) -- (w22);
			\draw (T0) -- (w23);
			\draw (T0) -- (w31);
			\draw (T0) -- (w32); 
			\draw (T0) -- (w33);
			
			\node at (-3.2,-0.9){$\bigcup_{i=3}^{s}T_i$};
			
			\coordinate   (A) at (-3.5,0.5);
			\coordinate   (B) at (-3.7,0.1);
			\coordinate   (C) at (-3.6,-0.4);
			\coordinate   (D) at (-1.8,-0.5);
			\coordinate   (E) at (-1.1,-0.5);
			\coordinate   (F) at (-1.2,0);
			\coordinate   (G) at (-1.3,0.5);
			\coordinate   (H) at (-1.7,0.5); 
			
			\draw[draw=none, fill=magenta!40] (A) to [closed, curve through = {(B) (C) (D) (E) (F) (G) (H)}] (A);
			
			\node[bnode] (a0) at (-1.5,0){};
			\node[bnode] (a1) at (-1.5,0.17){};
			\node[bnode] (a2) at (-1.4,0.36){};
			\node[bnode] (a3) at (-1.6,0.36){};
			\node[bnode] (a4) at (-1.7,0.1){};
			\node[bnode] (a5) at (-1.6,-0.2){};
			\node[bnode] (a6) at (-1.4,-0.18){};
			
			\node[bnode] (a7) at (-1.3,-0.4){};
			\node[bnode] (a8) at (-1.5,-0.4){};
			\node[bnode] (a9) at (-1.7,-0.4){};
			
			\foreach \t in {1,4,5,6}{
				\draw (a0)--(a\t);} 
			\draw (a1)--(a2);
			\draw (a1)--(a3);
			\draw (a6)--(a7);
			\draw (a5)--(a8);
			\draw (a5)--(a9);

			\node[bnode]  (H0) at (-3,0){};	 
			\node[bnode]  (H2) at ($(-3.2,0)+(135:0.13)$) {};
			\draw (H0)--(H2); 
			\node[bnode]  (H3) at ($(-3.2,0)+(210:0.13)$) {};
			\draw (H0)--(H3);   
			\node[bnode]  (I6) at ($(-3.2,0)+(210:0.27)$) {};  
			\node[bnode]  (I3) at ($(-3.2,0)+(130:0.27)$) {};
			\draw (H2)--(I3); 
			\draw (H3)--(I6);   
			\foreach \t in {3,4,...,8}{
				\node[bnode]  (J\t) at ($(-3.2,0)+(30*\t+15:0.46)$) {};}

			\draw (I3)--(J3); 	
			\draw (I3)--(J4);
			\draw (I3)--(J5);
			
			\draw (I6)--(J6); 	
			\draw (I6)--(J7);
			\draw (H0)--(J8);

			\node[bnode]  (P0) at (-2.8,0){};	  
			\node[bnode]  (P1) at ($(-2.8,0)+(60:0.13)$) {};
			\draw (P0)--(P1);  
			\node[bnode]  (P4) at ($(-2.8,0)+(320:0.13)$) {};
			\draw (P0)--(P4); 
			\foreach \t in {1,9}{
				\node[bnode]  (Q\t) at ($(-2.8,0)+(30*\t+30:0.27)$) {};}  
			\node[bnode]  (Q10) at ($(-2.8,0)+(340:0.27)$) {}; 
			\draw (P1)--(Q1); 	 
			\draw (P4)--(Q9);
			\draw (P4)--(Q10); 
			\foreach \t in {1,2,8,9,10,11,12}{
				\node[bnode]  (R\t) at ($(-2.8,0)+(30*\t+15:0.46)$) {};}
			\draw (Q1)--(R1); 	
			\draw (Q1)--(R2);
			
			\node at (-2,0){$\cdots$}; 
			\draw (P0)--(R12);
			
			\draw (Q9)--(R9); 	
			\draw (Q9)--(R10);
			\draw (Q10)--(R11);
			\draw (P4)--(R8);

			\node at ($(0,0)+(80:4.5 and 2.7)$){$a$}; 
			\node at ($(0,0)+(272:4.5 and 2.8)$){$b$}; 
			\node at ($(0,0)+(65:4.5 and 2.8)$){$u^1_1$}; 
			\node at ($(0,0)+(53:4.5 and 3)$){$u^1_{p_1}$};
			\node at ($(0,0)+(305:4.5 and 3)$){$u^k_1$}; 
			\node at ($(0,0)+(295:4.5 and 3)$){$u^k_{p_k}$}; 
			\node [rotate = 60] at ($(0,0)+(-10:3.5 and 2.9)$){$\cdots$}; 
			\node at ($(0,0)+(40:4.5 and 3)$){$v^1_1$}; 
			\node at ($(0,0)+(25:4.6 and 3.4)$){$v^1_{q_1}$}; 
			\node at ($(0,0)+(338:4.7 and 3.4)$){$v^{k-1}_1$};
			\node at ($(0,0)+(325:4.6 and 3.4)$){$v^{k-1}_{q_{k-1}}$};  
		\end{tikzpicture}  
		\caption{Case 1}
		\label{fig-case1} 
	\end{figure} 
	
	Let $G_1$ be the subgraph of $G$ induced by the edges in $L,T_{21},T_3,\ldots,T_s$ and $G_2$ be the subgraph of $G$ induced by the edges in $L,T_{22},T_1$. By induction hypothesis, there exists $\C^i \subseteq \C(G_i)$ which is feasible for $(G_i, L)$, $i=1,2$. We have the following.
	
	\begin{claimA}\label{claim-four cycles}
		The following statements hold.
		\begin{itemize}
			\item $\vv{C}_1=\vv{T_2}[b,a] \cup \vv{L}[a,b]$ is a dicycle of type $00$ with respect to $a$, of type $11$ with respect to $b$ in $\mathcal{C}^1$;
			\item $\vv{C}_2= \vv{T_2}[a,b] \cup \cevv{L}[a,b]$ is a dicycle of type $01$ with respect to $a$,  of type $10$ with respect to $b$ in $\mathcal{C}^1$;
			\item $\vv{C}_3 = \vv{T_2}[a,b] \cup \vv{L}[b,a]$ is a dicycle of type $11$ with respect to $a$,  of type $00$ with respect to $b$ in $\mathcal{C}^2$;
			\item $\vv{C}_4 = \vv{T_2}[b,a] \cup \cevv{L}[b,a]$ is a dicycle of type $10$ with respect to $a$,  of type $01$ with respect to $b$ in $\mathcal{C}^2$;
		\end{itemize} 
	\end{claimA}
	
	\begin{proof}
		We only prove the first statement, the rest can be proven similarly. By \ref{C4}, there exists a dicycle $\vv{C} \in \C^1$ which is of type $00$ with respect to $a$. Hence $(a,a^+)$ is an arc in $\vv{C}$. As in $G_1$, there is not attachments except $a$ and $b$ in $L[a,b]$, so by \ref{C3},  $\vv{C} = \vv{T_2}[b,a] \cup \vv{L}[a,b]$. It is clear that this dicycle is of type $11$ with respect to $b$.
	\end{proof}
	
	Then we let $\mathcal{C}=\mathcal{C}^1 \cup \mathcal{C}^2\setminus \{\vv{C}_1,\vv{C}_2,\vv{C}_3,\vv{C}_4\}$, where $\vv{C}_i$ is defined  as in Claim~\ref{claim-four cycles}. We now verify that $\mathcal{C}$ satisfies \ref{C1}-\ref{C4}.
	
	For each edge $e \in E(L)$, $|\mathcal{C}^1_e|=|\mathcal{C}^2_e|=2$, and it is also covered by exactly two  dicycles in $\{\vv{C}_1,\vv{C}_2,\vv{C}_3,\vv{C}_4\}$, so $|\mathcal{C}_e|=|\mathcal{C}^1_e|+|\mathcal{C}^2_e|-2=2$. Consider an arbitrary edge $e \in E(G) \setminus E(L)$. Since $E(T_2) = E(T_{21}) \cup E(T_{22})$ by Claim~\ref{claim-union}, we have $e \in E(G_1) \cup E(G_2)$. If $e \in E(G)\setminus [E(T_2[a,b]) \cup E(L)]$, then $e \in E(G_i) \setminus E(L)$ for some $ i \in \{1,2\}$ and $e \notin E(C_j)$ for any $j \in \{1,2,3,4\}$. Therefore $|\mathcal{C}_e| \geq |\mathcal{C}^i_e| \geq 4$, where the last inequality follows from \ref{C1} for $(G^i,L)$. If $e \in E(T_2[a,b])$, since $E(T_2[a,b]) \subseteq  E(T_{21}) \cap E(T_{22})$, and $e \in E(C_i)$ for each $i \in \{1,2,3,4\}$, then we can obtain $$|\mathcal{C}_e|=|\mathcal{C}^1_e|+|\mathcal{C}^2_e|-4 \geq 4+4-4 =4.$$ Therefore, \ref{C1} holds.
	
	Since each dicycle in $\mathcal{C}^1 \cup \mathcal{C}^2$  satisfies \ref{C2} by induction, $\mathcal{C}$ is a subset of $\mathcal{C}^1 \cup \mathcal{C}^2$, $\C$ satisfies \ref{C2} too.
	
	By Claim~\ref{claim-four cycles} and that \ref{C4} holds for $\C^1$, $\vv{C}_1$ and $\vv{C}_2$ are the only two dicycles in $\C^1$ containing vertices in $L[a^+,b^-]$. Thus every dicycle in $\mathcal{C}^1 \setminus \{\vv{C}_1, \vv{C}_2\}$ contains no attachments of $T_1$. On the other hand, as $\mathcal{C}^1$ satisfies \ref{C3} by induction hypothesis, every dicycle in $\mathcal{C}^1 \setminus \{\vv{C}_1, \vv{C}_2\}$ contains either no or two attachments of  $T_i$ for $i=2, 3, \ldots, s$. Similarly, we can verify that each dicycle in $\mathcal{C}^2 \setminus \{\vv{C}_3, \vv{C}_4\}$ contains no attachments of $T_i$ for $i =3,4,\ldots, s$ and contains either no or exactly two attachments of each of $T_1$ and $T_2$. Thus $\C$ satisfies \ref{C3}.
	
	For each $ij \in \mathbb{Z}^2_2$ and attachment $w$ distinct from $a$ or $b$, every dicycle of type $ij$ with respect to $w$ in $\mathcal{C}^1$ or $\mathcal{C}^2$ is kept in $\mathcal{C}$. For $a$ or $b$, there are exactly two dicycles of type $ij$ with respect to it in $\mathcal{C}^1 \cup \mathcal{C}^2$, but we removed one for each type by Claim \ref{claim-four cycles}. Thus,  $\mathcal{C}$ satisfies \ref{C4}.  
	
	\subsection{Exactly one of $u^1_1$ and $u^k_{p_k}$ witnesses attachments of $T_i$ for $3 \le i \le s$}
	\label{sec-case 2} 
	
	Without loss of generality, we assume that only $u^1_1$ witnesses a vertex $w$, which is an attachment of $T_i$ for some $3 \le i \le s$. Starting from $u^k_{p_k}$ and moving in the clockwise direction, let $b$ be the first attachment of $T_2$ that witnesses an attachment of $T_i$ for some $3 \le i \le s$. Let $x$ be the unique vertex in $T_2$ adjacent to $b$. We construct a new graph $G^*$ obtained from $G$ by subdividing the edge $ww^+$ to $waw^+$ and adding a new edge between $x$ and the new vertex $a$. See Fig.~\ref{fig-case2} for illustration. Let $L'$ be the corresponding subdivision of $L$. Define $T^*_2 = T_2 \cup xa$, and let $T^*_{21}$ be the subtree of $T^*_2$ whose leaves are precisely the attachments in $L'[b,a]$, and $T^*_{22}$ be the subtree of $T^*_2$ whose leaves are precisely the attachments in $L'[a,b]$. By arguments similar to those in Claim~\ref{claim-union}, we obtain the following.
	
	\begin{claimB}
		\label{claim2-union} 
		$E(T^*_2) = E(T^*_{21}) \cup E(T^*_{22})$. 
	\end{claimB}
	
	\begin{figure}[h]
		\centering
		\begin{tikzpicture}[>=latex,
			Tnode/.style={circle, draw= none, fill= magenta!50,opacity=0.3, minimum size=1.3cm, inner sep=0pt}, 
			bnode/.style={circle, draw=black,fill=green, minimum size=0.8mm, inner sep=0pt},
			snode/.style={circle, draw=black,fill=red, minimum size=0.8mm, inner sep=0pt},
			bluenode/.style={circle, draw=black,fill=red, minimum size=1.5mm, inner sep=0pt},
			rnode/.style={circle, draw=black,fill=orange, minimum size=0.8mm, inner sep=0pt},
			Bnode/.style={circle, draw=black,fill=white, minimum size=1.5mm, inner sep=0pt}] 
			
			\begin{scope}[yshift=4cm]  
				\draw ($(0,0)+(0:4.5 and 2.5)$) arc (0:360:4.5 and 2.5); 
				
				\node [Bnode] (u11) at ($(0,0)+(130:4.5 and 2.5)$){}; 
				\node [Bnode] (u12) at ($(0,0)+(135:4.5 and 2.5)$){};
				\node [Bnode] (u13) at ($(0,0)+(140:4.5 and 2.5)$){};
				
				\node [Bnode] (u31) at ($(0,0)+(30:4.5 and 2.5)$){};
				\node [Bnode] (u32) at ($(0,0)+(35:4.5 and 2.5)$){};
				\node [Bnode] (u33) at ($(0,0)+(40:4.5 and 2.5)$){}; 
				
				\node [Bnode] (u41) at ($(0,0)+(-30:4.5 and 2.5)$){};
				\node [Bnode] (u42) at ($(0,0)+(-35:4.5 and 2.5)$){};
				\node [Bnode] (u43) at ($(0,0)+(-40:4.5 and 2.5)$){}; 
				
				\node [Bnode] (u51) at ($(0,0)+(-77:4.5 and 2.5)$){};
				\node [Bnode] (u52) at ($(0,0)+(-82:4.5 and 2.5)$){};
				\node [Bnode] (u53) at ($(0,0)+(-87:4.5 and 2.5)$){}; 
				
				\node [Bnode] (u61) at ($(0,0)+(-130:4.5 and 2.5)$){};
				\node [Bnode] (u62) at ($(0,0)+(-135:4.5 and 2.5)$){};
				\node [Bnode] (u63) at ($(0,0)+(-140:4.5 and 2.5)$){}; 
				
				\node [Bnode] (v11) at ($(0,0)+(55:4.5 and 2.5)$){};
				\node [Bnode] (v12) at ($(0,0)+(60:4.5 and 2.5)$){};
				\node [Bnode] (v13) at ($(0,0)+(65:4.5 and 2.5)$){};

				\node [Bnode] (v21) at ($(0,0)+(0:4.5 and 2.5)$){};
				\node [Bnode] (v22) at ($(0,0)+(5:4.5 and 2.5)$){};
				\node [Bnode] (v23) at ($(0,0)+(10:4.5 and 2.5)$){};
				
				\node [Bnode] (v31) at ($(0,0)+(-55:4.5 and 2.5)$){};
				\node [Bnode] (v32) at ($(0,0)+(-65:4.5 and 2.5)$){}; 
				\node [Bnode] (v33) at ($(0,0)+(-60:4.5 and 2.5)$){};
				
				\node [Bnode] (w11) at ($(0,0)+(100:4.5 and 2.5)$){};
				\node [Bnode] (w12) at ($(0,0)+(105:4.5 and 2.5)$){};
				\node [Bnode] (w13) at ($(0,0)+(110:4.5 and 2.5)$){};
				\node [Bnode] (w+) at ($(0,0)+(80:4.5 and 2.5)$){};

				\node [Bnode] (w21) at ($(0,0)+(175:4.5 and 2.5)$){};
				\node [Bnode] (w22) at ($(0,0)+(180:4.5 and 2.5)$){};
				\node [Bnode] (w23) at ($(0,0)+(185:4.5 and 2.5)$){};
				
				\node [Bnode] (w31) at ($(0,0)+(-105:4.5 and 2.5)$){};
				\node [Bnode] (w32) at ($(0,0)+(-110:4.5 and 2.5)$){};
				\node [Bnode] (w33) at ($(0,0)+(-115:4.5 and 2.5)$){};

				\node [bluenode] (u21) at ($(0,0)+(90:4.5 and 2.5)$){};

				\node [snode] (T1) at (0.5,0){};
				\node at (1.2,-0.75){$T_2$};
				
				\draw (T1) -- (u11);
				\draw (0.3,0) -- (u12);
				\draw (0,0) -- (u13); 
				
				\draw (0.5,-0.1) -- (u31);
				\draw (T1) -- (u32);
				\draw (0.5,0.1) -- (u33);  
				
				\draw (T1) -- (u41);
				\draw (T1) -- (u42);
				\draw (T1) -- (u43); 
				
				\draw (T1) -- (u51);
				\draw (T1) -- (u52);

				\draw (T1) -- (u61);
				\draw (T1) -- (u62);
				\draw (T1) -- (u63);

				\coordinate   (Am) at ($(0.5,0)+(0:0.65)$);
				\coordinate   (Bm) at (0.5,-0.7);
				\coordinate   (Cm) at (0.3,-0.6);
				\coordinate   (Dm) at (0.1,-0.5);
				\coordinate   (Em) at (-0.1,-0.4);
				\coordinate   (Fm) at (-0.2,-0.1);
				\coordinate   (Gm) at (0.3,0.1);
				\coordinate   (Hm) at ($(0.5,0)+(90:0.6)$);
				
				\draw[draw=none, fill=brown!30] (Am) to [closed, curve through = {(Bm) (Cm) (Dm) (Em) (Fm) (Gm) (Hm)}] (Am);

				\node at (2.9,-0.5){$T_1$};  
				
				\node[snode]  (A0) at (0.5,0){};	 
				\node[snode]  (A1) at ($(0.5,0)+(45:0.15)$) {};  
				\node[snode]  (A3) at ($(0.5,0)+(210:0.15)$) {};
				\draw (A0)--(A1); 
				\draw (A0)--(A3); 
				\node[snode]  (A4) at ($(0.5,0)+(-30:0.15)$) {};
				\draw (A0)--(A4); 
				\foreach \t in {1,6,9}{
					\node[snode]  (B\t) at ($(0.5,0)+(30*\t+30:0.33)$) {};} 
				\node[snode]  (B10) at ($(0.5,0)+(340:0.33)$) {}; 
				\draw (A1)--(B1);   
				\draw (A3)--(B6);  
				\draw (A4)--(B9);
				\draw (A4)--(B10); 
				\foreach \t in {1,2,6,7,...,12}{
					\node[snode]  (C\t) at ($(0.5,0)+(30*\t+15:0.56)$) {};}
				\draw (B1)--(C1); 	
				\draw (B1)--(C2); 
				\draw (B6)--(C6); 	
				\draw (B6)--(C7); 
				\draw (A0)--(C8); 
				\draw (B9)--(C9); 	
				\draw (B9)--(C10);
				\draw (B10)--(C11);
				\draw (A4)--(C12);
	 
				\node [snode] (T2) at (2.5,0){}; 
				\draw (T2) -- (v11);
				\draw (T2) -- (v12);
				\draw (T2) -- (v13);
				
				\draw (T2) -- (v21);
				\draw (T2) -- (v22);
				\draw (T2) -- (v23);
				\draw (T2) -- (v31);
				\draw (T2) -- (v32); 
				\draw (T2) -- (v33);

				\coordinate   (Ar) at (2.6,0.52);
				\coordinate   (Br) at (2.2,0.25);
				\coordinate   (Cr) at (2,-0.2);
				\coordinate   (Dr) at (2.1,-0.4);
				\coordinate   (Er) at (2.3,-0.3);
				\coordinate   (Fr) at (2.5,-0.15);
				\coordinate   (Gr) at (2.8,-0.1);
				\coordinate   (Hr) at (3,0);
				
				\draw[draw=none, fill=cyan!30] (Ar) to [closed, curve through = {(Br) (Cr) (Dr) (Er) (Fr) (Gr) (Hr)}] (Ar);
				\node[rnode]  (D0) at (2.5,0){};	  
				\node[rnode]  (D1) at ($(2.5,0)+(60:0.13)$) {};
				\draw (D0)--(D1); 
				\node[rnode]  (D2) at ($(2.5,0)+(135:0.13)$) {};
				\draw (D0)--(D2); 
				\node[rnode]  (D3) at ($(2.5,0)+(210:0.13)$) {};
				\draw (D0)--(D3);

				\foreach \t in {1,6}{
					\node[rnode]  (E\t) at ($(2.5,0)+(30*\t+30:0.27)$) {};}
				
				\node[rnode]  (E3) at ($(2.5,0)+(130:0.27)$) {}; 
				
				\draw (D1)--(E1); 	
				\draw (D2)--(E3); 
				\draw (D3)--(E6);  
				
				\foreach \t in {1,2,6,7,12}{
					\node[rnode]  (F\t) at ($(2.5,0)+(30*\t+15:0.46)$) {};}
				\draw (E1)--(F1); 	
				\draw (E1)--(F2);  
				\draw (E6)--(F6); 	
				\draw (E6)--(F7); 
				\draw (D0)--(F12);  
		 
				\node [snode] (T0) at (-2,0){};	
				\draw (T0) -- (w11);
				\draw (T0) -- (w12);
				\draw (T0) -- (w13);
				\draw (T0) -- (w21);
				\draw (T0) -- (w22);
				\draw (T0) -- (w23);
				\draw (T0) -- (w31);
				\draw (T0) -- (w32); 
				\draw (T0) -- (w33);
				
				\node at (-3.2,-0.9){$\bigcup_{i=3}^{s}T_i$};
				
				\coordinate   (A) at (-3.5,0.5);
				\coordinate   (B) at (-3.7,0.1);
				\coordinate   (C) at (-3.6,-0.4);
				\coordinate   (D) at (-1.8,-0.5);
				\coordinate   (E) at (-1.1,-0.5);
				\coordinate   (F) at (-1.2,0);
				\coordinate   (G) at (-1.3,0.5);
				\coordinate   (H) at (-1.7,0.5); 
				
				\draw[draw=none, fill=magenta!40] (A) to [closed, curve through = {(B) (C) (D) (E) (F) (G) (H)}] (A);
			 
				\node[bnode] (a0) at (-1.5,0){};
				\node[bnode] (a1) at (-1.5,0.17){};
				\node[bnode] (a2) at (-1.4,0.36){};
				\node[bnode] (a3) at (-1.6,0.36){};
				\node[bnode] (a4) at (-1.7,0.1){};
				\node[bnode] (a5) at (-1.6,-0.2){};
				\node[bnode] (a6) at (-1.4,-0.18){};
				
				\node[bnode] (a7) at (-1.3,-0.4){};
				\node[bnode] (a8) at (-1.5,-0.4){};
				\node[bnode] (a9) at (-1.7,-0.4){};
				
				\foreach \t in {1,4,5,6}{
					\draw (a0)--(a\t);} 
				\draw (a1)--(a2);
				\draw (a1)--(a3);
				\draw (a6)--(a7);
				\draw (a5)--(a8);
				\draw (a5)--(a9);

				\node[bnode]  (H0) at (-3,0){};	 
				\node[bnode]  (H2) at ($(-3.2,0)+(135:0.13)$) {};
				\draw (H0)--(H2); 
				\node[bnode]  (H3) at ($(-3.2,0)+(210:0.13)$) {};
				\draw (H0)--(H3);   
				\node[bnode]  (I6) at ($(-3.2,0)+(210:0.27)$) {};  
				\node[bnode]  (I3) at ($(-3.2,0)+(130:0.27)$) {};
				\draw (H2)--(I3); 
				\draw (H3)--(I6);   
				\foreach \t in {3,4,...,8}{
					\node[bnode]  (J\t) at ($(-3.2,0)+(30*\t+15:0.46)$) {};}

				\draw (I3)--(J3); 	
				\draw (I3)--(J4);
				\draw (I3)--(J5);
				
				\draw (I6)--(J6); 	
				\draw (I6)--(J7);
				\draw (H0)--(J8);

				\node[bnode]  (P0) at (-2.8,0){};	  
				\node[bnode]  (P1) at ($(-2.8,0)+(60:0.13)$) {};
				\draw (P0)--(P1);  
				\node[bnode]  (P4) at ($(-2.8,0)+(320:0.13)$) {};
				\draw (P0)--(P4); 
				\foreach \t in {1,9}{
					\node[bnode]  (Q\t) at ($(-2.8,0)+(30*\t+30:0.27)$) {};}  
				\node[bnode]  (Q10) at ($(-2.8,0)+(340:0.27)$) {}; 
				\draw (P1)--(Q1); 	 
				\draw (P4)--(Q9);
				\draw (P4)--(Q10); 
				\foreach \t in {1,2,8,9,10,11,12}{
					\node[bnode]  (R\t) at ($(-2.8,0)+(30*\t+15:0.46)$) {};}
				\draw (Q1)--(R1); 	
				\draw (Q1)--(R2);
				
				\node at (-2,0){$\cdots$};
				\draw (P0)--(R12);
				
				\draw (Q9)--(R9); 	
				\draw (Q9)--(R10);
				\draw (Q10)--(R11);
				\draw (P4)--(R8);

				\draw (A0)--(u53);
				\draw[blue, thick] (A0) -- (u21);
				
				\node at ($(0,0)+(90:4.5 and 2.7)$){$a$}; 
				\node at ($(0,0)+(80:4.5 and 2.8)$){$w^+$}; 
				\node at ($(0,0)+(100:4.5 and 2.8)$){$w$}; 
				\node at (0.1,0.5){$x$};
				\node at ($(0,0)+(272:4.5 and 2.8)$){$b$}; 
				\node at ($(0,0)+(65:4.5 and 2.8)$){$u^1_1$}; 
				\node at ($(0,0)+(305:4.5 and 3)$){$u^k_1$}; 
				\node at ($(0,0)+(53:4.5 and 2.9)$){$u^1_{p_1}$};
				\node at ($(0,0)+(295:4.5 and 3)$){$u^k_{p_k}$}; 
				\node [rotate = 60] at ($(0,0)+(-10:3.5 and 2.9)$){$\cdots$};
				\node at ($(0,0)+(40:4.5 and 3)$){$v^1_1$}; 
				\node at ($(0,0)+(25:4.6 and 3.4)$){$v^1_{q_1}$}; 
				\node at ($(0,0)+(338:4.7 and 3.4)$){$v^{k-1}_1$};
				\node at ($(0,0)+(325:4.6 and 3.4)$){$v^{k-1}_{q_{k-1}}$};  
			\end{scope}   
			
			\draw[->,line width=2pt,red] (-4.5, 2.2) -- (-5, 1) node[midway, left, black] {$G^*_1$};
			\draw[->,line width=2pt,red] (4.5, 2.2) -- (5, 1) node[midway, right, black] {$G^*_2$};
			
			\begin{scope}[xshift=-4.5cm, yshift=-1.5cm] 
				\draw ($(0,0)+(0:3.2 and 2)$) arc (0:360:3.2 and 2); 
				
				\node [Bnode] (u11) at ($(0,0)+(100:3.2 and 2)$){}; 
				\node [Bnode] (u12) at ($(0,0)+(110:3.2 and 2)$){};
				\node [Bnode] (u13) at ($(0,0)+(120:3.2 and 2)$){};

				\node [Bnode] (u53) at ($(0,0)+(-30:3.2 and 2)$){}; 
				
				\node [Bnode] (u61) at ($(0,0)+(-110:3.2 and 2)$){};
				\node [Bnode] (u62) at ($(0,0)+(-120:3.2 and 2)$){};
				\node [Bnode] (u63) at ($(0,0)+(-130:3.2 and 2)$){};  
				
				\node [Bnode] (w11) at ($(0,0)+(60:3.2 and 2)$){};
				\node [Bnode] (w12) at ($(0,0)+(70:3.2 and 2)$){};
				\node [Bnode] (w13) at ($(0,0)+(80:3.2 and 2)$){};
				\node [Bnode] (w+) at ($(0,0)+(20:3.2 and 2)$){};

				\node [Bnode] (w21) at ($(0,0)+(170:3.2 and 2)$){};
				\node [Bnode] (w22) at ($(0,0)+(180:3.2 and 2)$){};
				\node [Bnode] (w23) at ($(0,0)+(190:3.2 and 2)$){};
				
				\node [Bnode] (w31) at ($(0,0)+(-60:3.2 and 2)$){};
				\node [Bnode] (w32) at ($(0,0)+(-70:3.2 and 2)$){};
				\node [Bnode] (w33) at ($(0,0)+(-80:3.2 and 2)$){};

				\node [bluenode] (u21) at ($(0,0)+(40:3.2 and 2)$){};  
				
				\node at (-2,-0.9){$\bigcup_{i=3}^{s}T_i$};
				
				\node [snode] (T1) at (1.5,0){};
				\node at (1.2,-0.95){$T^*_{21}$};
				
				\draw (1.2,0) -- (u11);
				\draw (1.1,0) -- (u12);
				\draw (1,0) -- (u13);

				\draw (T1) -- (u61);
				\draw (T1) -- (u62);
				\draw (T1) -- (u63); 
				  
				\coordinate   (Cm) at (1.5,0.3);
				\coordinate   (Dm) at (1.4,-0.6);
				\coordinate   (Em) at (1.1,-0.6);
				\coordinate   (Fm) at (0.9,0.1); 
				
				\draw[draw=none, fill=brown!30] (Cm) to [closed, curve through = {(Dm) (Em) (Fm) }] (Cm);
				
				\node[snode]  (A0) at (1.5,0){};	  
				\node[snode]  (A3) at ($(1.5,0)+(210:0.15)$) {};
				
				\draw (A0)--(A3);

				\node[snode]  (B6) at ($(1.5,0)+(210:0.33)$) {};
				\draw (A3)--(B6);  
				
				\foreach \t in {6,7,8}{
					\node[snode]  (C\t) at ($(1.5,0)+(30*\t+15:0.56)$) {};}
				
				\draw (B6)--(C6); 	
				\draw (B6)--(C7); 
				\draw (A0)--(C8);

				\draw (A0)--(u53);
				\draw[blue, thick] (A0) -- (u21);

				\coordinate   (Ar) at (2.6,0.52);
				\coordinate   (Br) at (2.2,0.25);
				\coordinate   (Cr) at (2,-0.2);
				\coordinate   (Dr) at (2.1,-0.4);
				\coordinate   (Er) at (2.3,-0.3);
				\coordinate   (Fr) at (2.5,-0.15);
				\coordinate   (Gr) at (2.8,-0.1);
				\coordinate   (Hr) at (3,0);

				\node [snode] (T0) at (-1,0){};	
				\draw (T0) -- (w11);
				\draw (T0) -- (w12);
				\draw (T0) -- (w13);
				\draw (T0) -- (w21);
				\draw (T0) -- (w22);
				\draw (T0) -- (w23);
				\draw (T0) -- (w31);
				\draw (T0) -- (w32); 
				\draw (T0) -- (w33);

				\coordinate   (A) at (-2.5,0.5);
				\coordinate   (B) at (-2.7,0.1);
				\coordinate   (C) at (-2.6,-0.4);
				\coordinate   (D) at (-0.8,-0.5);
				\coordinate   (E) at (-0.1,-0.5);
				\coordinate   (F) at (-0.2,0);
				\coordinate   (G) at (-0.3,0.5);
				\coordinate   (H) at (-0.7,0.5); 
				
				\draw[draw=none, fill=magenta!40] (A) to [closed, curve through = {(B) (C) (D) (E) (F) (G) (H)}] (A);

				\node[bnode] (a0) at (-0.5,0){};
				\node[bnode] (a1) at (-0.5,0.17){};
				\node[bnode] (a2) at (-0.4,0.36){};
				\node[bnode] (a3) at (-0.6,0.36){};
				\node[bnode] (a4) at (-0.7,0.1){};
				\node[bnode] (a5) at (-0.6,-0.2){};
				\node[bnode] (a6) at (-0.4,-0.18){};
				
				\node[bnode] (a7) at (-0.3,-0.4){};
				\node[bnode] (a8) at (-0.5,-0.4){};
				\node[bnode] (a9) at (-0.7,-0.4){};
				
				\foreach \t in {1,4,5,6}{
					\draw (a0)--(a\t);} 
				\draw (a1)--(a2);
				\draw (a1)--(a3);
				\draw (a6)--(a7);
				\draw (a5)--(a8);
				\draw (a5)--(a9);

				\node[bnode]  (H0) at (-2,0){};	 
				\node[bnode]  (H2) at ($(-2.2,0)+(135:0.13)$) {};
				\draw (H0)--(H2); 
				\node[bnode]  (H3) at ($(-2.2,0)+(210:0.13)$) {};
				\draw (H0)--(H3);   
				\node[bnode]  (I6) at ($(-2.2,0)+(210:0.27)$) {};  
				\node[bnode]  (I3) at ($(-2.2,0)+(130:0.27)$) {};
				\draw (H2)--(I3); 
				\draw (H3)--(I6);   
				\foreach \t in {3,4,...,8}{
					\node[bnode]  (J\t) at ($(-2.2,0)+(30*\t+15:0.46)$) {};}

				\draw (I3)--(J3); 	
				\draw (I3)--(J4);
				\draw (I3)--(J5);
				
				\draw (I6)--(J6); 	
				\draw (I6)--(J7);
				\draw (H0)--(J8);

				\node[bnode]  (P0) at (-1.8,0){};	  
				\node[bnode]  (P1) at ($(-1.8,0)+(60:0.13)$) {};
				\draw (P0)--(P1);  
				\node[bnode]  (P4) at ($(-1.8,0)+(320:0.13)$) {};
				\draw (P0)--(P4); 
				\foreach \t in {1,9}{
					\node[bnode]  (Q\t) at ($(-1.8,0)+(30*\t+30:0.27)$) {};}  
				\node[bnode]  (Q10) at ($(-1.8,0)+(340:0.27)$) {}; 
				\draw (P1)--(Q1); 	 
				\draw (P4)--(Q9);
				\draw (P4)--(Q10); 
				\foreach \t in {1,2,8,9,10,11,12}{
					\node[bnode]  (R\t) at ($(-1.8,0)+(30*\t+15:0.46)$) {};}
				\draw (Q1)--(R1); 	
				\draw (Q1)--(R2);
				
				\node at (-2,0){$\cdots$};
				\draw (P0)--(R12);
				
				\draw (Q9)--(R9); 	
				\draw (Q9)--(R10);
				\draw (Q10)--(R11);
				\draw (P4)--(R8);

				\node at ($(0,0)+(42:3.5 and 2.2)$){$a$}; 
				\node at ($(0,0)+(20:3.7 and 2.2)$){$w^+$}; 
				\node at ($(0,0)+(60:3.6 and 2.2)$){$w$}; 
				\node at (1.3,0.2){$x$};
				\node at ($(0,0)+(-30:3.5 and 2.2)$){$b$};  
			\end{scope}  
			
			\begin{scope}[xshift=4.5cm, yshift=-1.5cm] 
				\draw ($(0,0)+(0:3.2 and 2)$) arc (0:360:3.2 and 2); 
				
				\node [Bnode] (u31) at ($(0,0)+(60:3.2 and 2)$){};
				\node [Bnode] (u32) at ($(0,0)+(70:3.2 and 2)$){};
				\node [Bnode] (u33) at ($(0,0)+(80:3.2 and 2)$){}; 
				
				\node [Bnode] (u41) at ($(0,0)+(-50:3.2 and 2)$){};
				\node [Bnode] (u42) at ($(0,0)+(-60:3.2 and 2)$){};
				\node [Bnode] (u43) at ($(0,0)+(-70:3.2 and 2)$){}; 
				
				\node [Bnode] (u51) at ($(0,0)+(-150:3.2 and 2)$){};
				\node [Bnode] (u52) at ($(0,0)+(-160:3.2 and 2)$){};
				\node [Bnode] (u53) at ($(0,0)+(-170:3.2 and 2)$){}; 
				
				\node [Bnode] (v11) at ($(0,0)+(100:3.2 and 2)$){};
				\node [Bnode] (v12) at ($(0,0)+(110:3.2 and 2)$){};
				\node [Bnode] (v13) at ($(0,0)+(120:3.2 and 2)$){};

				\node [Bnode] (v21) at ($(0,0)+(0:3.2 and 2)$){};
				\node [Bnode] (v22) at ($(0,0)+(10:3.2 and 2)$){};
				\node [Bnode] (v23) at ($(0,0)+(20:3.2 and 2)$){};
				
				\node [Bnode] (v31) at ($(0,0)+(-90:3.2 and 2)$){};
				\node [Bnode] (v32) at ($(0,0)+(-110:3.2 and 2)$){}; 
				\node [Bnode] (v33) at ($(0,0)+(-100:3.2 and 2)$){}; 
				
				\node [Bnode] (w11) at ($(0,0)+(165:3.2 and 2)$){}; 
				\node [Bnode] (w+) at ($(0,0)+(135:3.2 and 2)$){};
				
				\node [bluenode] (u21) at ($(0,0)+(150:3.2 and 2)$){};

				\node [snode] (T1) at (-0.5,0){};
				\node at (-1,-0.8){$T^*_{22}$};

				\draw (T1) -- (u31);
				\draw (T1) -- (u32);
				\draw (T1) -- (u33);  
				
				\draw (T1) -- (u41);
				\draw (T1) -- (u42);
				\draw (T1) -- (u43); 
				
				\draw (T1) -- (u51);
				\draw (T1) -- (u52);

				\coordinate   (Bm) at (-1,-0.5);
				\coordinate   (Cm) at (-1.1,0.5);
				\coordinate   (Dm) at (-1,0.6);
				\coordinate   (Em) at (-0.6,0.6);
				\coordinate   (Fm) at (-0.3,-0.3);
				\coordinate   (Gm) at (-0.5,-0.6); 
				
				\draw[draw=none, fill=brown!30] (Bm) to [closed, curve through = {(Cm) (Dm) (Em) (Fm) (Gm)}] (Bm);

				\node at (1.8,-0.5){$T_1$};  
				
				\node[snode]  (A0) at (-1,0){};	 
				\node[snode]  (A1) at ($(-1,0)+(45:0.15)$) {};  
				\node[snode]  (A3) at ($(-1,0)+(210:0.15)$) {};
				\draw (A0)--(A1); 
				\draw (A0)--(A3); 
				\node[snode]  (A4) at ($(-1,0)+(-30:0.15)$) {};
				\draw (A0)--(A4); 
				\foreach \t in {1,9}{
					\node[snode]  (B\t) at ($(-1,0)+(30*\t+30:0.33)$) {};} 
				\node[snode]  (B10) at ($(-1,0)+(340:0.33)$) {}; 
				\draw (A1)--(B1);

				\draw (A4)--(B9);
				\draw (A4)--(B10); 
				\foreach \t in {1,2,9,10,11,12}{
					\node[snode]  (C\t) at ($(-1,0)+(30*\t+15:0.56)$) {};}
				\draw (B1)--(C1); 	
				\draw (B1)--(C2); 
				
				\draw (B9)--(C9); 	
				\draw (B9)--(C10);
				\draw (B10)--(C11);
				\draw (A4)--(C12);
			 
				\node [snode] (T2) at (1.5,0){}; 
				\draw (T2) -- (v11);
				\draw (T2) -- (v12);
				\draw (T2) -- (v13);
				
				\draw (T2) -- (v21);
				\draw (T2) -- (v22);
				\draw (T2) -- (v23);
				\draw (T2) -- (v31);
				\draw (T2) -- (v32); 
				\draw (T2) -- (v33); 
				 
				\coordinate   (Ar) at (1.6,0.52);
				\coordinate   (Br) at (1.2,0.25);
				\coordinate   (Cr) at (1,-0.2);
				\coordinate   (Dr) at (1.1,-0.4);
				\coordinate   (Er) at (1.3,-0.3);
				\coordinate   (Fr) at (1.5,-0.15);
				\coordinate   (Gr) at (1.8,-0.1);
				\coordinate   (Hr) at (2,0);
				
				\draw[draw=none, fill=cyan!30] (Ar) to [closed, curve through = {(Br) (Cr) (Dr) (Er) (Fr) (Gr) (Hr)}] (Ar);
				\node[rnode]  (D0) at (1.5,0){};	  
				\node[rnode]  (D1) at ($(1.5,0)+(60:0.13)$) {};
				\draw (D0)--(D1); 
				\node[rnode]  (D2) at ($(1.5,0)+(135:0.13)$) {};
				\draw (D0)--(D2); 
				\node[rnode]  (D3) at ($(1.5,0)+(210:0.13)$) {};
				\draw (D0)--(D3);

				\foreach \t in {1,6}{
					\node[rnode]  (E\t) at ($(1.5,0)+(30*\t+30:0.27)$) {};}
				
				\node[rnode]  (E3) at ($(1.5,0)+(130:0.27)$) {}; 
				
				\draw (D1)--(E1); 	
				\draw (D2)--(E3); 
				\draw (D3)--(E6);  
				
				\foreach \t in {1,2,6,7,12}{
					\node[rnode]  (F\t) at ($(1.5,0)+(30*\t+15:0.46)$) {};}
				\draw (E1)--(F1); 	
				\draw (E1)--(F2);  
				\draw (E6)--(F6); 	
				\draw (E6)--(F7); 
				\draw (D0)--(F12);

				\draw (A0)--(u53);
				\draw[blue, thick] (A0) -- (u21);
				
				\node at ($(0,0)+(150:3.5 and 2.2)$){$a$}; 
				\node at ($(0,0)+(135:3.5 and 2.2)$){$w^+$}; 
				\node at ($(0,0)+(165:3.5 and 2.2)$){$w$}; 
				\node at (-1.1,0.3){$x$};
				\node at (-3.5,-0.3){$b$};  
				\node at (-3.2,-1.2){$v^k_1$};  
				\node at ($(0,0)+(120:3.5 and 2.4)$){$u^1_1$}; 
				\node at ($(0,0)+(100:3.5 and 2.4)$){$u^1_{p_1}$};
				\node at ($(0,0)+(270:3.5 and 2.4)$){$u^k_1$}; 
				\node at ($(0,0)+(250:3.5 and 2.5)$){$u^k_{p_k}$}; 
				\node [rotate = 20] at (2.3,-0.8){$\cdots$};
				\node at ($(0,0)+(80:3.5 and 2.3)$){$v^1_1$}; 
				\node at ($(0,0)+(310:3.6 and 2.4)$){$v^{k-1}_1$};
				\node at ($(0,0)+(60:3.6 and 2.4)$){$v^1_{q_1}$}; 
				\node at ($(0,0)+(290:3.6 and 2.4)$){$v^{k-1}_{q_{k-1}}$}; 
			\end{scope}
		\end{tikzpicture}  
		\caption{Illustration for the construction of $G^*$, $G^*_1$ and $G^*_2$.}
		\label{fig-case2} 
	\end{figure}
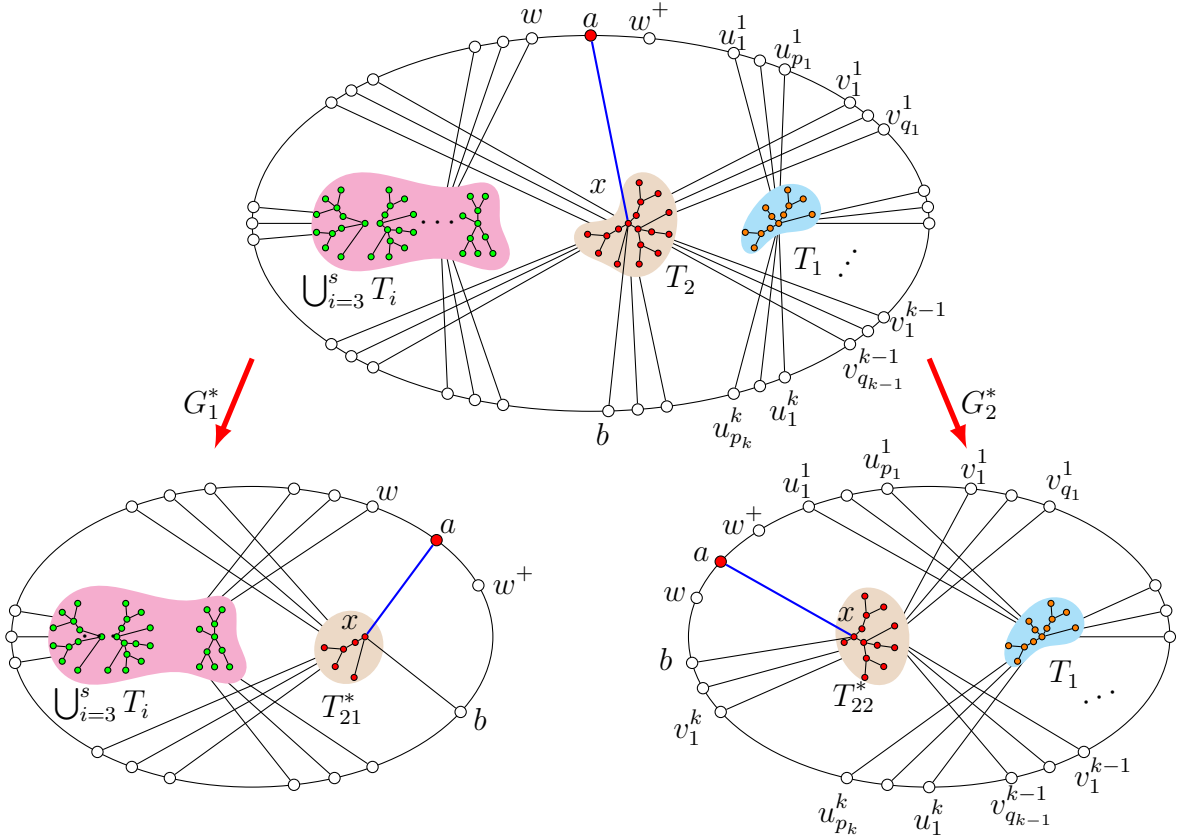
	
	Let $G^*_1$ be the subgraph of $G^*$ induced by the edges in $L',T^*_{21},T_3,\ldots,T_s$ and $G^*_2$ be the subgraph of $G^*$ induced by the edges in $L',T^*_{22},T_1$. See again Fig.~\ref{fig-case2} for illustration. By induction hypothesis, there exists $\C^1 \in \C(G^*_1)$ that is feasible for $(G^*_1, L')$. Instead of applying induction on $(G^*_2, L')$, we construct  $\C^2$ that is feasible for $(G^*_2, L')$, following the construction in the basic step with $s=2$. More precisely, $T^*_{22}$ plays the role of $T_2$, $a$ plays the role of $v^k_{q_k}$ and $b$ plays the role of $v^k_{q_k-1}$, and the remaining attachments of $T^*_{22}$ along $L[u^k_{p_k}, b^-]$ play the role of $v^k_1, v^k_2, \ldots, v^k_{q_k-2}$ in order, respectively.
	
	Note that both $a$ and $b$ are two attachments if $T^*_{2i}$ in $G^*_i$ for $i=1,2$. 
	By \ref{C4} for $(G^*_1, L')$ and $(G^*_2,L')$, for any $ij \in \mathbb{Z}^2_2$ and $t \in \{1,2\}$, there is a unique dicycle in $\C^t$ which is of type $ij$ with respect to $a$, we denote such dicycle as $\vv{C}^{ij}_t(a)$, and denote the other attachment of $T_2$ contained in $\vv{C}^{ij}_t(a)$ as $a^{ij}_t$. We denote $\vv{C}^{ij}_t(b)$ similarly. Also, we denote by $\C^i(a)$  all the dicycles in $\C^i$ containing $a$ for $i=1,2$. The following claim lists six of the dicycles in $\C^1(a) \cup \C^2(a)$ and some properties of the rest two.
	\begin{claimB}\label{claim2-six cycles}
		The following statements hold.
		\begin{enumerate}[label= {(\arabic*)}] 
			\item $\vv{C}^{00}_1 (a) = \vv{C}^{11}_1(b) = (b,x) \cup (x,a) \cup \vv{L}[a,b]$.
			\item $\vv{C}^{01}_1 (a) = \vv{C}^{10}_1(b) = (a,x) \cup (x,b) \cup \cevv{L}[a,b]$.
			\item $\vv{C}^{10}_2 (a) = \vv{C}^{01}_2(b) = (b,x) \cup (x,a) \cup \cevv{L}[b,a]$.
			\item $\vv{C}^{11}_2 (a) = \vv{C}^{00}_2(b) = (a,x) \cup (x,b) \cup \vv{L}[b,a]$.
			\item $\vv{C}^{00}_2 (a) = \vv{T^*_2}[v^{k-1}_{q_{k-1}}, a] \cup \vv{L'}[a,u^1_1] \cup \vv{T_1}[u^1_1, u^k_1] \cup \cevv{L'}[ v^{k-1}_{q_{k-1}}, u^k_1]$. (Hence $a^{00}_2=v^{k-1}_{q_{k-1}}$) 
			\item $\vv{C}^{01}_2 (a) = \vv{T^*_2}[a, v^1_{q_1}] \cup \vv{L'}[v^1_{q_1}, u^2_1] \cup \vv{T_1}[u^2_1, u^1_1] \cup \cevv{L'}[a, u^1_1]$. (Hence $a^{01}_2=v^1_{q_1}$)
			\item Neither $\vv{C}^{11}_1(a)$ nor $\vv{C}^{10}_1(a)$ contain the vertex $b$.
		\end{enumerate} 
	\end{claimB}
	\begin{proof}
		The first two can be proven as in the proof of Claim~\ref{claim-four cycles}. The third through sixth statements follow from the construction of $\C^2$. We focus on the last one, and only prove that $\vv{C}^{11}_1(a)$ does not contain $b$, the other part can be shown similarly. Suppose to the contrary, $b$ is contained in $\vv{C}^{11}_1(a)$. Since $\vv{C}^{11}_1(a)$ is of type $11$ with respect to $a$, we know that $L'[a,b]$ is not contained in $\vv{C}^{11}_1(a)$, hence the cycle $L'[a,b] \cup T^*_{21}[a,b]$ is contained in  $L' \triangle C^{11}_1(a)$. As $\C^1$ holds for \ref{C2}, it must be that $L' \triangle C^{11}_1(a) = L'[a,b] \cup T^*_{21}[a,b]$. This implies that $\vv{C}^{11}_1(a) = L'[b,a] \cup T^*_{21}[a,b]$. Hence $\vv{C}^{11}_1(a)$ contains all the attachments of $T^*_{21}$. Since $T_2$ overlaps with some $T_i$ for $i \in \{3,4,\ldots, s\}$, $T^*_{21}$ must contain another attachment other than $a$ and $b$. So $\vv{C}^{11}_1(a)$ contains at least three attachments of $T^*_{21}$, which violates \ref{C3}.  
	\end{proof}

	Now we are going to construct $\C$ that is feasible for $(G,L)$. To do this, we introduce several auxiliary subgraphs and notations that will be used in the construction.

	Given two dicycles $\vv{C}_1$ and $\vv{C}_2$ such that their intersection $C_1 \cap C_2$ is a path containing at least one edge, and such that the orientations of $\vv{C}_1$ and $\vv{C}_2$ on this common path are opposite, we define $\vv{C}_1 \triangle \vv{C}_2$ to be the digraph whose underlying graph is $C_1 \triangle C_2$, with arc directions inherited from $\vv{C}_1 \cup \vv{C}_2$. It is clear that $\vv{C}_1 \triangle \vv{C}_2$ is also a dicycle. 
	
	Let $Q_1 = T_2[x,a^{11}_1] \cap T_2[x, a^{00}_2]$ and $Q_2 = T_2[x, a^{10}_1] \cap T_2[x, a^{01}_2]$. Note that both $Q_1$ and $Q_2$ are paths in $T_2$, and may consist of only the single vertex $x$. Since $C^{11}_1(a)$ and $C^{00}_2(a)$ share common edges and vertices only within $T_2$, it follows that $$C^{11}_1(a) \cap C^{00}_2(a) = xa \cup Q_1.$$ Similarly, we have that $$C^{10}_1(a) \cap C^{01}_2(a) = xa \cup Q_2.$$ On the other hand, by the definitions of $\vv{C}^{11}_1(a)$ and $\vv{C}^{00}_2(a)$, the orientations of these two dicycles are opposite on $xa \cup Q_1$. Thus $\vv{C}^{11}_1(a) \triangle \vv{C}^{00}_2(a)$ is also a dicycle. Let $\vv{C}^{\text{new}}_1$ be the dicycle obtained from $\vv{C}^{11}_1(a) \triangle \vv{C}^{00}_2(a)$ by contracting $w \rightarrow a \rightarrow w^+$ to $w \rightarrow w^+$. Similarly, let $\vv{C}^{\text{new}}_2$ be the dicycle obtained from $\vv{C}^{10}_1(a) \triangle \vv{C}^{01}_2(a)$ by contracting $w^+ \rightarrow a \rightarrow w$ to $w^+ \rightarrow w$. See Fig.~\ref{fig:new dicycle} for illustration.
	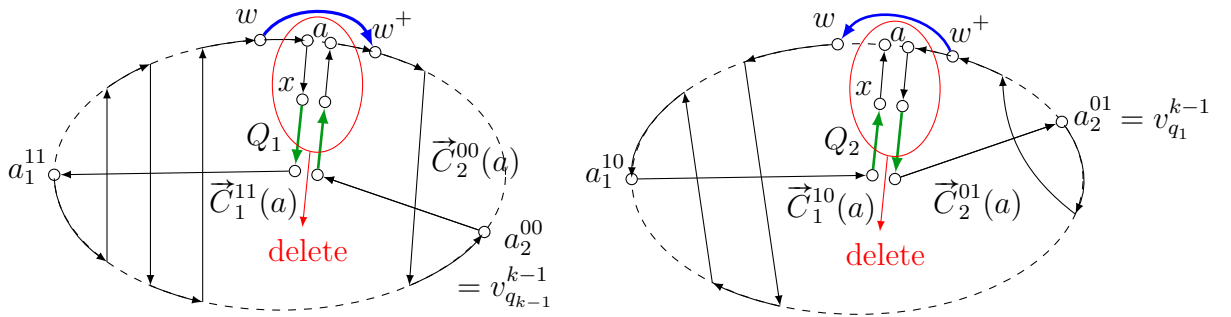
\begin{figure}[h]
		\centering
		\begin{minipage}{0.45\textwidth} 
			\begin{tikzpicture}[>=latex,
				Tnode/.style={circle, draw= none, fill = yellow, opacity=0.3, minimum size=7mm, inner sep=0pt},
				snode/.style={circle, draw=black,fill=white, minimum size=1.5mm, inner sep=0pt}] 
				
				\draw[dashed] (0,0) ellipse (3cm and 1.8cm); 
				\node[snode, label={[xshift=-6pt, yshift= -3pt] $x$}] (x) at (0.3,1) {};
				\node[snode] (x') at (0.6,0.97) {};
				
				\node[snode, label={[xshift=-5pt] $w$}] (w) at (95:3cm and 1.8cm) {};
				\node[snode, label={[xshift=5pt, yshift=-5pt] $a$}] (a) at (83:3cm and 1.8cm) {};
				\node[snode] (a') at (77:3cm and 1.8cm) {};
				\node[snode] (w+) at (65:3cm and 1.8cm) {}; 
				
				\node at (1.5,2) {$w^+$};
				
				\node[snode] (y) at (0.2,0.05) {};
				\node[snode] (y') at (0.5,0) {};
				
				\node[snode, label={[xshift=-10pt,yshift=-10pt] $a^{11}_1$}] (a111) at (180:3cm and 1.8cm) {};
				
				\draw [->] (a)-- (x);
				\draw [darkgreen,line width=1pt,->] (x)--(y); 
				\draw [->] (y)--(a111); 
				\draw [->] (181.5: 3cm and 1.8cm) arc (181.5: 220: 3cm and 1.8cm); 
				\draw [->] (220:3cm and 1.8cm)--(140:3cm and 1.8cm); 
				\draw [->] (140: 3cm and 1.8cm) arc (140: 125: 3cm and 1.8cm);
				\draw [->] (125:3cm and 1.8cm)--(235:3cm and 1.8cm); 
				\draw [->] (235: 3cm and 1.8cm) arc (235: 250: 3cm and 1.8cm);
				\draw [->] (250:3cm and 1.8cm)--(110:3cm and 1.8cm); 
				\draw [->] (110: 3cm and 1.8cm) arc (110: 96: 3cm and 1.8cm);
				\draw [->] (94: 3cm and 1.8cm) arc (94: 83.5: 3cm and 1.8cm);
				
				\node[snode, label={[xshift=15pt,yshift=-15pt] $a^{00}_2$}] (a002) at (-25:3cm and 1.8cm) {};
				\node at (3,-1.5) {$=v^{k-1}_{q_{k-1}}$}; 
				
				\draw [darkgreen,line width=1pt,->] (y')-- (x');
				\draw [->] (x')--(a'); 
				\draw [->] (a002)--(y'); 
				\draw [->] (75: 3cm and 1.8cm) arc (75: 65.5: 3cm and 1.8cm); 
				\draw [->] (63.5: 3cm and 1.8cm) arc (63.5: 50: 3cm and 1.8cm); 
				\draw [->] (50:3cm and 1.8cm)--(-55:3cm and 1.8cm); 
				\draw [->] (-55: 3cm and 1.8cm) arc (-55: -26: 3cm and 1.8cm);
				\draw [->] (a002)--(y');  
				
				\draw [blue, line width=1.2pt,->] (w) [bend left = 60] to (w+);
				
				\draw[red] (0.5,1.2) ellipse (0.6cm and 0.9cm);
				\draw[red,->] (0.4,0.3) to (0.3,-0.7);  
				\node[red, right] at (-0.3,-1) {delete}; 
				
				\node at (-0.2,0.5) {$Q_1$}; 
				\node at (-0.35,-0.35) {$\vv{C}^{11}_1(a)$};  
				\node at (2.6,0.2) {$\vv{C}^{00}_2(a)$}; 
			\end{tikzpicture} 
		\end{minipage}
		\begin{minipage}{0.45\textwidth} 
			\begin{tikzpicture}[>=latex,
				Tnode/.style={circle, draw= none, fill = yellow, opacity=0.3, minimum size=7mm, inner sep=0pt},
				snode/.style={circle, draw=black,fill=white, minimum size=1.5mm, inner sep=0pt}] 
				
				\draw[dashed] (0,0) ellipse (3cm and 1.8cm); 
				\node[snode, label={[xshift=-6pt, yshift= -3pt] $x$}] (x) at (0.3,1) {};
				\node[snode] (x') at (0.6,0.97) {};
				
				\node[snode, label={[xshift=-5pt] $w$}] (w) at (95:3cm and 1.8cm) {};
				\node[snode, label={[xshift=5pt, yshift=-5pt] $a$}] (a) at (83:3cm and 1.8cm) {};
				\node[snode] (a') at (77:3cm and 1.8cm) {};
				\node[snode] (w+) at (65:3cm and 1.8cm) {}; 
				
				\node at (1.5,2) {$w^+$};
				
				\node[snode] (y) at (0.2,0.05) {};
				\node[snode] (y') at (0.5,0) {};
				
				\node[snode, label={[xshift=-10pt,yshift=-10pt] $a^{10}_1$}] (a111) at (180:3cm and 1.8cm) {};
				
				\draw [<-] (a)-- (x);
				\draw [darkgreen,line width=1pt,<-] (x)--(y); 
				\draw [<-] (y)--(a111); 
				\draw [<-] (179.5: 3cm and 1.8cm) arc (179.5: 140: 3cm and 1.8cm);
				\draw [<-] (140:3cm and 1.8cm)--(230:3cm and 1.8cm);   
				\draw [<-] (230: 3cm and 1.8cm) arc (230: 250: 3cm and 1.8cm);
				\draw [<-] (250:3cm and 1.8cm)--(120:3cm and 1.8cm);  
				\draw [<-] (120: 3cm and 1.8cm) arc (120: 96.5: 3cm and 1.8cm);
				
				\node[snode, label={[xshift=12pt,yshift=-10pt] $a^{01}_2$}] (a012) at (25:3cm and 1.8cm) {};
				\node at (4.1,0.8) {$=v^{k-1}_{q_1}$}; 
				
				\draw [darkgreen,line width=1pt,<-] (y')-- (x');
				\draw [<-] (x')--(a'); 
				\draw [<-] (a012)--(y'); 
				\draw [<-] (75: 3cm and 1.8cm) arc (75: 65.5: 3cm and 1.8cm); 
				\draw [<-] (63.5: 3cm and 1.8cm) arc (63.5: 50: 3cm and 1.8cm); 
				\draw [<-] (50:3cm and 1.8cm) [bend right = 25] to (-15:3cm and 1.8cm); 
				\draw [<-] (-15: 3cm and 1.8cm) arc (-15: 23.5: 3cm and 1.8cm);
				\draw [<-] (a012)--(y');  
				
				\draw [blue, line width=1pt,<-] (w) [bend left = 60] to (w+);
				
				\draw[red] (0.5,1.2) ellipse (0.6cm and 0.9cm);
				\draw[red,->] (0.4,0.3) to (0.3,-0.7);  
				\node[red, right] at (-0.3,-1) {delete}; 
				
				\node at (-0.2,0.5) {$Q_2$}; 
				\node at (-0.35,-0.35) {$\vv{C}^{10}_1(a)$};  
				\node at (1.6,-0.3) {$\vv{C}^{01}_2(a)$}; 
			\end{tikzpicture} 
		\end{minipage}
		\caption{Illustration for the construction of  $\protect\vv{C}^{\protect\text{new}}_1$ (left) and $\protect\vv{C}^{\protect\text{new}}_2$ (right). The underlying graph of the green diptahs in the left figure are both $Q_1$, while those that in the right figure are both $Q_2$.}
		\label{fig:new dicycle} 
	\end{figure}

	Now we construct $\C$ as follows, let $$\mathcal{C} = \left(\mathcal{C}^1\cup \mathcal{C}^2\setminus [\mathcal{C}^1(a)\cup\mathcal{C}^2(a)]\right) \cup \{\vv{C}^{\text{new}}_1, \vv{C}^{\text{new}}_2\}.$$
	
	In the rest of this subsection, we check that $\C$ satisfies \ref{C1}-\ref{C4}. 
	
	We first show that $\C$ satisfies \ref{C1}. It is not difficult to verify that $|\mathcal{C}_e| = 2$ for each $e \in E(L)$. We then focus on edges in $E(G) \setminus E(L)$. 
	
	As each edge in $Q_1$ is also in $T^*_{21}[b,a^{11}_1] \cap T^*_{22}[b, a^{00}_2]$, and each edge in $Q_2$ is also in $T^*_{21}[b,a^{10}_1] \cap T^*_{22}[b, a^{01}_2]$, the following holds. 
	\begin{claimB}
		\label{claim-Q1Q2}
		For each edge $e \in E(Q_1) \cup E(Q_2) \cup \{xb\}$, we have $e \in E(T^*_{21} \cap T^*_{22})$.  
	\end{claimB} 
	
	Assume $e \in E(G) \setminus E(L)$, clearly $e \neq xa$. Then we consider the following cases. 
	\begin{itemize}
		\item $e = xb$. Note that by the first four items of Claim~\ref{claim2-six cycles}, $e$ is contained in each of the cycles $C^{00}_1(a)$, $C^{01}_1(a)$, $C^{10}_2(a)$, $C^{11}_2(a)$. By the last three items of Claim~\ref{claim2-six cycles}, $e$ is not contained in each of $C^{10}_1(a)$,  $C^{11}_1(a)$, $C^{00}_2(a)$, $C^{01}_2(a)$. On the other hand, by Claim~\ref{claim-Q1Q2}, $e \in E(G^*_1) \cap E(G^*_2)$. Therefore,  $|\C_e| = |\C^1_e| + |\C^2_e|-4 \geq 4$.
		\item $e \in E(Q_1) \cup E(Q_2)$. By Claim~\ref{claim-Q1Q2}, $e \in E(G^*_1) \cap E(G^*_2)$. On the other hand, $e$ is not contained in any of the four dicycles $\vv{C}^{00}_1(a)$, $\vv{C}^{01}_1(a)$, $\vv{C}^{10}_2(a)$, $\vv{C}^{11}_2(a)$, but it is possible that $e$ is in $C^{11}_1(a) \cap C^{00}_2(a)$ or $C^{10}_1(a) \cap C^{01}_2(a)$ as $e \in E(Q_1) \cup E(Q_2)$. Hence $|\C_e| \geq  |\C^1_e| + |\C^2_e|-4 \geq 4$.
		\item $e \in E(G) \setminus [E(L) \cup E(Q_1) \cup E(Q_2) \cup \{bx\}]$. Note that in this case, it is not possible that $e \in E(C^{11}_1(a)) \cup E(C^{00}_2(a))$ but $e \notin E(C^{\text{new}}_1)$ as $e \notin E(Q_1)$. Similarly, it is not possible that $e \in E(C^{10}_1(a)) \cup E(C^{01}_2(a))$ but $e \notin E(C^{\text{new}}_2)$ as $e \notin E(Q_2)$. Then by Claim~\ref{claim2-union},  $e$ is contained in  $E(G^*_i) \setminus (E(L') \cup \{xa\})$ for some $i \in \{1,2\}$. On the other hand, $e$ is not contained in any of the four dicycles $\vv{C}^{00}_1(a)$, $\vv{C}^{01}_1(a)$, $\vv{C}^{10}_2(a)$, $\vv{C}^{11}_2(a)$. Thus, as $\C^i$ is feasible for $(G^*_i, L')$, we have $|\C_e| \geq |\C^i_e| \geq 4$. 
	\end{itemize}   
	Therefore, we proved that $\mathcal{C}$ satisfies \ref{C1}.
	
	By our construction of $\mathcal{C}$, to check that $\mathcal{C}$ satisfies \ref{C2}, it suffices to check that both $C^{\text{new}}_1 \triangle L$ and $C^{\text{new}}_2 \triangle L$ are cycles. 
	Observe that  $C^{11}_1(a) \cap C^{00}_2(a) = Q_1 \cup xa$ and $C^{10}_1(a) \cap C^{01}_2(a) = Q_2 \cup xa$. Since $C^{ij}_t(a) \triangle L'$ is a cycle for $ij \in \mathbb{Z}^2_2$, $t \in \{1,2\}$, we decompose four of them to edge-disjoint paths as follows.  
	\begin{align*} 
		C^{11}_1(a) \triangle L' &  = xa \cup Q_1 \cup L'[a,b] \cup R_1, \\
		C^{00}_2(a) \triangle L' &  = xa \cup Q_1 \cup L'[b,a] \cup R_2, \\ 
		C^{10}_1(a) \triangle L' &  = xa \cup Q_2 \cup L'[a,b] \cup R_3, \\
		C^{01}_2(a) \triangle L' &  = xa \cup Q_2 \cup L'[b,a] \cup R_4,
	\end{align*}   
	where $R_1$ and $R_3$ are in $G^*_1$, $R_2$ and $R_4$ are in $G^*_2$. In addition, $R_1$ and $R_3$ share only the vertices $b$ and one endpoint of $Q_1$ distinct from $x$, and similarly, $R_2$ and $R_4$ share only  the vertices $b$ and one endpoint of $Q_2$ distinct from $x$. It follows that $C^{\text{new}}_1 \triangle L = R_1 \cup R_2$ and $C^{\text{new}}_2 \triangle L = R_3 \cup R_4$ are both cycles. Thus $\mathcal{C}$ satisfies \ref{C2}.
	
	For \ref{C3}, we only need to focus on $C^{\text{new}}_1$ and $C^{\text{new}}_2$, the rest can be verified similarly as in Section~\ref{sec-case 1}. It is clear that $C^{\text{new}}_1$ contains two attachments of $T_2$, which are $a^{11}_1$ and $a^{00}_2$, and contains either no or two attachments of $T_j$ for $j\neq 2$ as \ref{C3} holds for $C^{11}_1(a)$ and $C^{00}_2(a)$. Similarly, we can verify that $C^{\text{new}}_2$ satisfies \ref{C3}.
	
	Now we show that $\mathcal{C}$ satisfies \ref{C4}. This statement holds for $b$ by the similar arguments as in Section~\ref{sec-case 1}. Thus we assume $w$ is an attachment distinct from $b$. If $w$ is not contained in any dicycle in $\mathcal{C}^1(a) \cup \mathcal{C}^2(a)$, then each dicycle containing $w$ is kept in $\mathcal{C}$. So we assume that $w$ is contained in at least one dicycle in $\mathcal{C}^1(a) \cup \mathcal{C}^2(a)$. By Claim~\ref{claim2-six cycles}, $w$ is not contained in any of $\vv{C}^{00}_1(a)$, $\vv{C}^{01}_1(a)$, $\vv{C}^{11}_2(a)$ and $\vv{C}^{10}_2(a)$, each of which contains two attachments $a$ and $b$ of $T^*_2$. If $w$ is contained in $\vv{C}^{11}_1(a)$ (resp., $\vv{C}^{00}_2(a)$) which is the unique dicycle of  type $ij$ with respect to $w$ for some $ij \in \mathbb{Z}^2_2$, then $\vv{C}^{\text{new}}_1$ is also the unique dicycle of type $ij$ with respect to $w$. If $w$ is contained in $\vv{C}^{10}_1(a)$ (resp., $\vv{C}^{00}_2(a)$) which is the unique dicycle of  type $ij$ with respect to $w$ for some $ij \in \mathbb{Z}^2_2$, then  $\vv{C}^{\text{new}}_2$ is also the unique dicycle of type $ij$ with respect to $w$.

	\subsection{Both $u^1_1$ and $u^k_{p_k}$ witness attachments of $T_i$ for $3 \le i \le s$}\label{case3}
	
	In this case, let $w$ be the first attachment that $u^1_1$ witnesses when moving in the counterclockwise direction, and let $z$ be the first attachment that $u^k_{p_k}$ witnesses when moving in the clockwise direction.  
	Here, $w$ and $z$ are attachments of $T_i$ and $T_j$, respectively, for some $3 \le i,j \le s$ (possibly with $i=j$).

	We choose a non-leaf vertex $x$ in $T_2$ as follows. If $v^1_1 = v^1_{q_1} = v^{k-1}_1 = v^{k-1}_{q_{k-1}}$ (when $k=2$ and $q_1=1$), then let $x$ be the unique neighbor of $v^1_{q_1}$ in $T_2$. Otherwise, let $x$ be a non-leaf vertex in the subtree of $T_2$ whose leaves are exactly $\{v^1_1, v^1_{q_1}, v^{k-1}_1, v^{k-1}_{q_{k-1}}\}$.  
	We choose $x$ so that its degree in this subtree is as large as possible. In addition, if $v^1_{q_1} \neq v^{k-1}_{q_{k-1}}$, then make sure that $x$ is in $T_2[v^1_{q_1}, v^{k-1}_{q_{k-1}}]$. See Fig.~\ref{fig:case3} for illustration.
	
	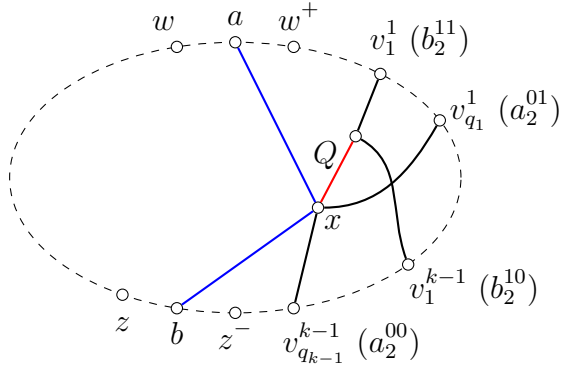
\begin{figure}  
		\centering
		\begin{tikzpicture}[>=latex,
			Tnode/.style={circle, draw= none, fill = yellow, opacity=0.3, minimum size=7mm, inner sep=0pt},
			snode/.style={circle, draw=black,fill=white, minimum size=1.5mm, inner sep=0pt}] 
			
			\draw[dashed] (0,0) ellipse (3cm and 1.8cm);
			
			\node[snode, label={[xshift=-5pt] $w$}] (w) at (105:3cm and 1.8cm) {};
			
			\node[snode, label={[xshift=0pt] $a$}] (a) at (90:3cm and 1.8cm) {};
			\node[snode, label={[xshift=2pt] $w^+$}] (w+) at (75:3cm and 1.8cm) {}; 
			
			\node[snode, label={[xshift=15pt] $v^1_1$ ($b^{11}_2$)}] (v11) at (50:3cm and 1.8cm) {};
			\node[snode, label={[xshift=25pt,yshift=-10pt] $v^1_{q_1}$ ($a^{01}_2$)}] (v1q) at (25:3cm and 1.8cm) {}; 
			
			\node[snode, label={[xshift=25pt, yshift=-22pt] $v^{k-1}_1$ ($b^{10}_2$)}] (vk1) at (-40:3cm and 1.8cm) {};
			\node[snode, label={[xshift=22pt, yshift=-28pt] $v^{k-1}_{q_{k-1}}$ ($a^{00}_2$)}] (vkq) at (-75:3cm and 1.8cm) {};
			
			\node[snode, label={[yshift=-20 pt] $z$}] (z) at (-120:3cm and 1.8cm) {};
			\node[snode, label={[yshift=-20 pt] $b$}] (b) at (-105:3cm and 1.8cm) {}; 
			\node[snode, label={[yshift=-20 pt] $z^-$}] (z-) at (-90:3cm and 1.8cm) {};

			\node [snode] (y) at (1.6,0.55) {};
			\node [snode] (x) at (1.1,-0.4) {};
			\node at (1.3,-0.6) {$x$};
			
			\draw[thick] (v11) to (y);
			\draw[red, thick] (y) to (x);
			\draw[thick] (x) to (vkq); 
			\node at (1.2,0.3) {$Q$};
			
			\draw[thick] (v1q) to[out=240, in = 0] (x);
			\draw[thick] (y) to[out=330, in = 110] (vk1);
			
			\draw [blue, thick] (a)--(x)--(b);
		\end{tikzpicture} 
		\caption{Illustration for the choice of $x$ and $xa, xb$. The red path is $Q$.}
		\label{fig:case3} 
	\end{figure} 
	
	We construct a new graph $G^*$ which obtained from $G$ by subdividing $ww^+$ to $waw^+$ and adding a new edge between $x$ and the new vertex $a$, and subdividing $zz^-$ to $zbz^-$ and adding an edge between $x$ and the new vertex $b$.  Let $L'$ be the corresponding subdivision of $L$. Let $T^*_2, T^*_{21}, T^*_{22}, G^*_1, G^*_2$, $\C^1, \C^2, \ \vv{C}^{ij}_t(a), a^{ij}_t$ be defined or obtained as last section. Similarly, we define $\vv{C}^{ij}_t(b)$ and $b^{ij}_t$ for $b$ as $\vv{C}^{ij}_t(a)$ and $a^{ij}_t$, respectively. Obviously, Claim~\ref{claim2-union} also holds here.

	Let $Q_1 = T_2[x,a^{11}_1] \cap T_2[x, a^{00}_2]$, $Q_2 = T_2[x, a^{10}_1] \cap T_2[x, a^{01}_2]$, $Q_3 = T_2[x, b^{00}_1] \cap T_2[x, b^{11}_2]$ and $T_2[x, b^{10}_1] \cap T_2[x, b^{01}_2]$. The following is obviously true. 
	\begin{claimC}
		\label{claim-Q1-Q4}
		Every edge in $E(Q_1) \cup E(Q_2) \cup E(Q_3) \cup E(Q_4) \cup \{xa, xb\}$ is also in $E(T^*_{21} \cap T^*_{22})$.  
	\end{claimC} 
	
	By \ref{C3} and \ref{C4} for $\C^1$ and the construction of $\C^2$, Claim~\ref{claim2-six cycles} also holds here. In addition, we have the following.
	\begin{claimC}\label{claim-8 cycles}
		The following statements hold.
		\begin{enumerate}[label= {(\arabic*)}]
			\item $\vv{C}^{10}_2 (b) = \vv{T^*_2}[v^{k-1}_1, b] \cup \cevv{L'}[u^k_{p_k},b] \cup \vv{T_1}[u^k_{p_k}, u^{k-1}_{p_{k-1}}] \cup \vv{L'}[u^{k-1}_{p_{k-1}}, v^{k-1}_1]$. (Hence $b^{10}_2=v^{k-1}_1$) 
			\item $\vv{C}^{11}_2 (b) = \vv{T^*_2}[b, v^1_1] \cup \cevv{L'}[u^1_{p_1}, v^1_1] \cup \vv{T_1}[u^1_{p_1}, u^k_{p_k}] \cup \vv{L'}[u^k_{p_k},b]$. (Hence $b^{11}_2=v^1_1$) 
		\end{enumerate}
	\end{claimC}
	
	Thus by Claim~\ref{claim2-six cycles} and Claim~\ref{claim-8 cycles}, $a^{00}_2=v^{k-1}_{q_{k-1}}$, $a^{01}_2=v^1_{q_1}$, $b^{10}_2=v^{k-1}_1$ and $b^{11}_2=v^1_1$. Let $T_a$ be the subtree of $T_2$ whose leaves are precisely $\{a^{00}_2, a^{01}_2, a^{10}_1, a^{11}_1\}$, and let $T_b$ be the subtree of $T_2$ whose leaves are precisely $\{b^{00}_1, b^{01}_1, b^{10}_2, b^{11}_2\}$. We claim the following.
	\begin{claimC}
		\label{Q1capQ2}
		$E(Q_1) \cap E(Q_2) \cap E(T_a)= \emptyset$ and $E(Q_3) \cap E(Q_4) \cap E(T_b)= \emptyset$. In particular,
		if $x$ is in $T_a$, then $E(Q_1) \cap E(Q_2) = \emptyset$. Similarly, if $x$ is in $T_b$, then $E(Q_3) \cap E(Q_4) = \emptyset$.
	\end{claimC}
	\begin{proof}
		Assume to the contrary, $e \in E(Q_1) \cap E(Q_2)$ and also $e \in E(T_a)$. Then $T_a-e$ contains two components, say $T_{a1}$ and $T_{a2}$. Let $y$ be the vertex in $Q_1 \cap Q_2 \cap T_a$ that is closest to $x$ along $Q_1$. Clearly, $y$ is also in $Q_2$ for otherwise, there is a cycle in $T_2$.  Without loss of generality, assume $y \in V(T_{a1})$. By the definition of $Q_1$ and $Q_2$, it must be that all of $a^{00}_2, a^{01}_2, a^{10}_1, a^{11}_1$ are contained in $T_{a2}$ as $e \in \in E(Q_1) \cap E(Q_2)$. However, by definition, $T_a$ is the minimal subtree of $T_2$ containing all of these four vertices. This implies that $T_a \subseteq T_{a2}$ and hence does not contain $y$, which contradicts the assumption that $y \in V(T_{a1})$.
		
		Similarly, we can prove $E(Q_3) \cap E(Q_4) \cap E(T_b)= \emptyset$.
		
		Note that if $x$ is in $T_a$, and $E(Q_1) \cap E(Q_2) \neq \emptyset$, then $E(Q_1) \cap E(Q_2) \cap E(T_a) \neq \emptyset$, a contradiction. Similarly, we can show that if $x$ is in $T_b$, then $E(Q_3) \cap E(Q_4) = \emptyset$.
	\end{proof} 
	
	\begin{claimC}
		\label{claim-at most 2}
		Each edge in $E(G)$ is contained in at most two of $Q_1, Q_2, Q_3, Q_4$.
	\end{claimC}
	\begin{proof}
		First assume that $v^{k-1}_{q_{k-1}} = v^1_{q_1}$, it follows that $a^{00}_2=a^{01}_2$ (by the fifth and sixth items of Claim~\ref{claim2-six cycles}), hence $T_a$ has at most three leaves. Thus at least one of $Q_1$ and $Q_2$ is empty. Without loss of generality, assume $Q_1=\emptyset$. If $v^1_1 = v^1_{q_1} = v^{k-1}_1 = v^{k-1}_{q_{k-1}}$, then it is easy to see that $Q_1 = Q_2 = Q_3 = Q_4 = \emptyset$, the claim is obviously true. Otherwise, by the choice of $x$, we know that $x$ must lie in $T_2[v^1_1, v^{k-1}_1]$ and hence in $T_b$. By Claim~\ref{Q1capQ2}, $E(Q_3) \cap E(Q_4) = \emptyset$. Together with the fact that $Q_1=\emptyset$, the claim holds.
		
		Assume that $v^{k-1}_{q_{k-1}} \neq v^1_{q_1}$. By the choice of $x$, we have that $x$ is in $T_2[v^1_{q_1}, v^{k-1}_{q_{k-1}}] = T_2[a^{01}_2, a^{01}_2]$, hence in $T_a$. By Claim~\ref{Q1capQ2}, $E(Q_1) \cap E(Q_2) = \emptyset$. 
		
		If $x$ is also in $T_b$, then again by Claim~\ref{Q1capQ2}, $E(Q_3) \cap E(Q_4) = \emptyset$. Thus there is no edge could be in three of $E(Q_1), E(Q_2), E(Q_3), E(Q_4)$, the claim follows. 
		
		Assume that $x$ is not contained in $T_b$. Then, by the choice of $x$, the subpaths $T_2[a^{00}_2, a^{01}_2]$ and $T_2[b^{10}_2, b^{11}_2]$ are vertex-disjoint, see Fig.~\ref{fig:case3} for illustration. Since $T_2$ is a tree, there exists a unique path $Q$ in $T_2$ joining these two subpaths. By the choice of $x$, it must be one endpoint of $Q$, and we denote the other endpoint by $y$. Clearly, $y$ is in $T_2[b^{10}_2, b^{11}_2]$, and hence it is also in $T_b$. Most importantly, we know that $Q_3 \cap Q_4 = Q$ by Claim~\ref{Q1capQ2} and the definition of $Q_3, Q_4, Q$. 
		
		By the choice of $x$,  both $Q_1 \subset T_2[a^{00}_2, a^{01}_2]$ and $Q_2 \subset T_2[a^{00}_2, a^{01}_2]$, but $Q$ is edge-disjoint to $T_2[a^{00}_2, a^{01}_2]$, so we have 
		$$E(Q_3) \cap E(Q_4) \cap (E(Q_1) \cup E(Q_2)) =  E(Q) \cap (E(Q_1) \cup E(Q_2)) = \emptyset.$$ 
		Together with the fact that $E(Q_1) \cap E(Q_2) = \emptyset$, the claim follows.
	\end{proof}
	
	Now we shall construct $\C$ that is feasible for $(G,L)$. Let $$\mathcal{C} = \left(\mathcal{C}^1\cup \mathcal{C}^2\setminus [\mathcal{C}^1(a)\cup\mathcal{C}^2(a) \cup \mathcal{C}^1(b)\cup\mathcal{C}^2(b)]\right) \cup \{\vv{C}^{\text{new}}_1, \vv{C}^{\text{new}}_2, \vv{C}^{\text{new}}_3, \vv{C}^{\text{new}}_4\}, $$ where $\vv{C}^{\text{new}}_1$ and $\vv{C}^{\text{new}}_2$ are defined as last subsection, $\vv{C}^{\text{new}}_3$ and $\vv{C}^{\text{new}}_4$ are defined as follows.
	\begin{itemize}
		\item $\vv{C}^{\text{new}}_3$ is obtained from $\vv{C}^{00}_1(b) \triangle \vv{C}^{11}_2(b)$ by contracting $z^- \rightarrow b \rightarrow z$ to $z^- \rightarrow z$.
		\item $\vv{C}^{\text{new}}_4$ is obtained from $\vv{C}^{01}_1(b) \triangle \vv{C}^{10}_2(b)$ by contracting $z \rightarrow b \rightarrow z^-$ to $z \rightarrow z^-$.
	\end{itemize}

	The verification of $\mathcal{C}$ satisfying \ref{C2}-\ref{C4} is analogous to the last subsection, except that there are additional dicycles $\vv{C}^{\text{new}}_3$ and $\vv{C}^{\text{new}}_4$, which need to be verified for \ref{C2}-\ref{C3}, and one additional vertex $b$ to be considered for \ref{C4}. We therefore omit the repeated arguments. We can also deduce that $|\C_e|=2$ for $e \in E(L)$ and $|\mathcal{C}_e| \geq 4$ for each $e \in E(G) \setminus [E(L) \cup E(T^*_{21} \cap T^*_{22})]$ as in that part. Thus it remains to verify \ref{C1} for the edges in $\left[E(T^*_{21} \cap T^*_{22})\right] \setminus \{xa,xb\}$.

	By Claims~\ref{claim2-six cycles},~\ref{claim-8 cycles}, and the definition of $Q_1, Q_2, Q_3, Q_4$, we have the following.
	\begin{align*}
		C^{11}_1(a) \cap C^{00}_2(a) &= xa \cup Q_1, \quad C^{10}_1(a) \cap C^{01}_2(a)  = xa \cup Q_2, \\
		C^{00}_1(b) \cap C^{11}_2(b) &= xb \cup Q_3, \quad C^{01}_1(b) \cap C^{10}_2(b)  = xb \cup Q_4.
	\end{align*} 
	Therefore, by Claims~\ref{claim-Q1-Q4},~\ref{claim-at most 2}, $|\C_e| \geq	|\C^1_e|+|\C^2_e|-4 \geq 4$ for each edge $e \in \left[E(T^*_{21} \cap T^*_{22})\right] \setminus \{xa,xb\}$. 
	
	This completes the proof of Theorem~\ref{main-thm2}. \qed

	\section{Concluding remarks} 
	
	As shown in Voss~\cite{Voss1991}, the case $k \le 3$ of Conjecture~\ref{main-conj} 
	provides a useful estimate of the number of edges in certain subgraphs of the 
	bridges of a longest cycle, which has proved valuable in many problems on cycles.
	Hence, we would expect that resolving this conjecture (i.e., Theorem~\ref{thm:main}) may lead to further applications in the study of cycles.

	One potential application of Theorem~\ref{thm:main} concerns the size of the intersection 
	of two longest cycles in highly connected graphs. A well-known conjecture 
	(see \cite{GGL,Grotschel1984}), often attributed to Scott Smith, asserts the following.
	
	\begin{conjecture}\label{conj-intersection}
		In $k$-connected graphs, any two longest cycles intersect in at least $k$ vertices.
	\end{conjecture}
	
	\noindent There has been extensive research on this conjecture; see 
	\cite{CFG1998,Grotschel1984,shabbir2013} and recent results \cite{GLSTY2024,MZ25}. 
	Let $C$ and $D$ be two longest cycles in a 2-connected graph $G$ (so that they intersect 
	in at least two vertices), and let $H = C \cup D$ denote the 2-connected subgraph formed by their union. 
	The $C$-bridges of $H$ are all subpaths of $D$, and existing approaches to Conjecture~\ref{conj-intersection} 
	often analyze how these bridges are arranged along $C$. 
	In this context, Conjecture~\ref{conj-intersection} aligns closely with the essence of Conjecture~\ref{main-conj}, 
	and our main result may provide a useful tool for studying this structure in more detail.
	This conjecture is also closely related to the famous conjecture of Lov{\'a}sz 
	on the circumference of vertex-transitive graphs 
	(see \cite{Babai1979,DeVos2023,GLSTY2024,MZ25,NSTW25} for details).
	
	Another potential application of Theorem~\ref{thm:main} is related to a problem of Babai~\cite{Babai1979}, 
	which asks about the intersection size of two longest cycles in a 3-connected cubic graph.
	
	\begin{problem}[Babai~\cite{Babai1979}, Problem 2]\label{problem-Babai}
		Let $f(c)$ denote the largest integer with the following property: if a 3-connected graph has circumference $c$, then any two longest cycles of the graph intersect in at least $f(c)$ vertices. 
		Does $f(c) \to \infty$ as $c \to \infty$? 
	\end{problem}
	
	\noindent Suppose $C$ and $D$ are two longest cycles of a 3-connected cubic graph $G$, and let $H = C \cup D$, which is a 2-connected subgraph of $G$. 
	Since $G$ is cubic, all $C$-bridges of $H$, say $B_1,\dots,B_k$, are vertex-disjoint subpaths of $D$, whose union is exactly $E(D) \setminus E(C)$, so that
	\(
	\sum_{i=1}^k \lambda(B_i) = |E(D) \setminus E(C)|.
	\)
	Now suppose that $O_H(C)$ forms a tree. By Theorem~\ref{thm:main}, we then have 
	\(
	\sum_{i=1}^k \lambda(B_i) \le |E(C)|/2.
	\)
	Combining these, we obtain
	\[
	|E(C \cap D)| = |E(D)| - |E(D) \setminus E(C)| = |E(C)| - \sum_{i=1}^k \lambda(B_i) \ge |C|/2.
	\]
	This bound is certainly too strong to hold in general. 
	A more practical approach is to analyze the overlap graph of $O_H(C)$ and attempt to decompose it into small trees, 
	so that Theorem~\ref{thm:main} can be applied to each tree in a meaningful way.

	\section*{Acknowledgments}
	The authors thank Yufan Luo, Zihan Zhou, and Ziyuan Zhao for carefully reading the draft 
	and providing valuable comments that improved the presentation.

	\bibliographystyle{abbrv}
	\bibliography{reference}

\begin{thebibliography}{10}

\bibitem{AG1985}
B.~R. Alspach and C.~D. Godsil, editors.
\newblock {\em Cycles in graphs}, volume 115 of {\em North-Holland Mathematics
  Studies}.
\newblock North-Holland Publishing Co., Amsterdam, 1985.
\newblock Papers from the workshop held at Simon Fraser University, Burnaby,
  B.C., July 5--August 20, 1982, Annals of Discrete Mathematics, 27.

\bibitem{AP1961}
L.~Auslander and S.~V. Parter.
\newblock On imbedding graphs in the sphere.
\newblock {\em J. Math. Mech.}, 10:517--523, 1961.

\bibitem{Babai1979}
L.~Babai.
\newblock Long cycles in vertex-transitive graphs.
\newblock {\em J. Graph Theory}, 3(3):301--304, 1979.

\bibitem{Bondy2014}
J.~A. Bondy.
\newblock Beautiful conjectures in graph theory.
\newblock {\em European J. Combin.}, 37:4--23, 2014.

\bibitem{BE1980}
J.~A. Bondy and R.~C. Entringer.
\newblock Longest cycles in {$2$}-connected graphs with prescribed maximum
  degree.
\newblock {\em Canadian J. Math.}, 32(6):1325--1332, 1980.

\bibitem{BM2008}
J.~A. Bondy and U.~S.~R. Murty.
\newblock {\em Graph theory}, volume 244 of {\em Graduate Texts in
  Mathematics}.
\newblock Springer, New York, 2008.

\bibitem{CFG1998}
G.~Chen, R.~J. Faudree, and R.~J. Gould.
\newblock Intersections of longest cycles in {$k$}-connected graphs.
\newblock {\em J. Combin. Theory Ser. B}, 72(1):143--149, 1998.

\bibitem{CSYZ2006}
G.~Chen, L.~Sheppardson, X.~Yu, and W.~Zang.
\newblock The circumference of a graph with no {$K_{3,t}$}-minor.
\newblock {\em J. Combin. Theory Ser. B}, 96(6):822--845, 2006.

\bibitem{CY2002}
G.~Chen and X.~Yu.
\newblock Long cycles in 3-connected graphs.
\newblock {\em J. Combin. Theory Ser. B}, 86(1):80--99, 2002.

\bibitem{CYZ2012}
G.~Chen, X.~Yu, and W.~Zang.
\newblock The circumference of a graph with no {$K_{3,t}$}-minor, {II}.
\newblock {\em J. Combin. Theory Ser. B}, 102(6):1211--1240, 2012.

\bibitem{CN1986}
N.~Chiba and T.~Nishizeki.
\newblock A theorem on paths in planar graphs.
\newblock {\em J. Graph Theory}, 10(4):449--450, 1986.

\bibitem{DeVos2023}
M.~DeVos.
\newblock Long cycles in vertex-transitive graphs.
\newblock {\em arXiv:2302.04255v1}, 2023.

\bibitem{Goldstein1963}
A.~J. Goldstein.
\newblock An efficient and constructive algorithm for testing whether a graph
  can be embedded in a plane.
\newblock In {\em Graph and Combinatorics Conference}, Princeton, NJ, May 1963.
  Office of Naval Research Logistics Project, Department of Mathematics,
  Princeton University.
\newblock Contract No. NONR 1858-(21), May 16--18, 2 pp.

\bibitem{GGL}
R.~L. Graham, M.~Gr\"{o}tschel, and L.~Lov\'{a}sz, editors.
\newblock {\em Handbook of combinatorics. {V}ol. 1, 2}.
\newblock Elsevier Science B.V., Amsterdam; MIT Press, Cambridge, MA, 1995.

\bibitem{GLSTY2024}
C.~Groenland, S.~Longbrake, R.~Steiner, J.~Turcotte, and L.~Yepremyan.
\newblock Longest cycles in vertex-transitive and highly connected graphs.
\newblock {\em Bulletin of the London Mathematical Society},
  https://doi.org/10.1112/blms.70134, 2025.

\bibitem{Grotschel1984}
M.~Gr\"{o}tschel.
\newblock On intersections of longest cycles.
\newblock In {\em Graph theory and combinatorics ({C}ambridge, 1983)}, pages
  171--189. Academic Press, London, 1984.

\bibitem{HT1974}
J.~Hopcroft and R.~Tarjan.
\newblock Efficient planarity testing.
\newblock {\em J. ACM}, 21(4):549–568, Oct. 1974.

\bibitem{JW1992}
B.~Jackson and N.~C. Wormald.
\newblock Longest cycles in 3-connected planar graphs.
\newblock {\em Journal of Combinatorial Theory, Series B}, 54(2):291--321,
  1992.

\bibitem{JY2002}
B.~Jackson and X.~Yu.
\newblock Hamilton cycles in plane triangulations.
\newblock {\em J. Graph Theory}, 41(2):138--150, 2002.

\bibitem{KNZ2007}
K.~Kawarabayashi, J.~Niu, and C.-Q. Zhang.
\newblock Chords of longest circuits in locally planar graphs.
\newblock {\em European J. Combin.}, 28(1):315--321, 2007.

\bibitem{KO2015}
K.~Kawarabayashi and K.~Ozeki.
\newblock 4-connected projective-planar graphs are {H}amiltonian-connected.
\newblock {\em J. Combin. Theory Ser. B}, 112:36--69, 2015.

\bibitem{MZ25}
J.~Ma and Z.~Zhao.
\newblock Intersections of longest cycles in vertex-transitive and highly
  connected graphs.
\newblock {\em arXiv:2508.17438}, 2025.

\bibitem{NSTW25}
S.~Norin, R.~Steiner, S.~Thomass{\'e}, and P.~Wollan.
\newblock Small hitting sets for longest paths and cycles.
\newblock {\em arXiv preprint arXiv:2505.08634}, 2025.

\bibitem{OZ2018}
K.~Ozeki and C.~T. Zamfirescu.
\newblock Every 4-connected graph with crossing number 2 is {H}amiltonian.
\newblock {\em SIAM J. Discrete Math.}, 32(4):2783--2794, 2018.

\bibitem{Sanders1997}
D.~P. Sanders.
\newblock On paths in planar graphs.
\newblock {\em J. Graph Theory}, 24(4):341--345, 1997.

\bibitem{shabbir2013}
A.~Shabbir, C.~T. Zamfirescu, and T.~I. Zamfirescu.
\newblock Intersecting longest paths and longest cycles: A survey.
\newblock {\em Electronic Journal of Graph Theory and Applications (EJGTA)},
  1(1):56--76, 2013.

\bibitem{TY1994}
R.~Thomas and X.~Yu.
\newblock {$4$}-connected projective-planar graphs are {H}amiltonian.
\newblock {\em J. Combin. Theory Ser. B}, 62(1):114--132, 1994.

\bibitem{TY1997}
R.~Thomas and X.~Yu.
\newblock Five-connected toroidal graphs are {H}amiltonian.
\newblock {\em J. Combin. Theory Ser. B}, 69(1):79--96, 1997.

\bibitem{TYZ2005}
R.~Thomas, X.~Yu, and W.~Zang.
\newblock Hamilton paths in toroidal graphs.
\newblock {\em J. Combin. Theory Ser. B}, 94(2):214--236, 2005.

\bibitem{Thomassen1983}
C.~Thomassen.
\newblock A theorem on paths in planar graphs.
\newblock {\em J. Graph Theory}, 7(2):169--176, 1983.

\bibitem{Thomassen1989}
C.~Thomassen.
\newblock Configurations in graphs of large minimum degree, connectivity, or
  chromatic number.
\newblock In {\em Combinatorial {M}athematics: {P}roceedings of the {T}hird
  {I}nternational {C}onference ({N}ew {Y}ork, 1985)}, volume 555 of {\em Ann.
  New York Acad. Sci.}, pages 402--412. New York Acad. Sci., New York, 1989.

\bibitem{Thomassen1997}
C.~Thomassen.
\newblock Chords of longest cycles in cubic graphs.
\newblock {\em J. Combin. Theory Ser. B}, 71(2):211--214, 1997.

\bibitem{Tutte1956}
W.~T. Tutte.
\newblock A theorem on planar graphs.
\newblock {\em Trans. Amer. Math. Soc.}, 82:99--116, 1956.

\bibitem{Tutte1959}
W.~T. Tutte.
\newblock Matroids and graphs.
\newblock {\em Trans. Amer. Math. Soc.}, 90:527--552, 1959.

\bibitem{Voss1991}
H.-J. Voss.
\newblock {\em Cycles and bridges in graphs}, volume~49 of {\em Mathematics and
  its Applications (East European Series)}.
\newblock Kluwer Academic Publishers Group, Dordrecht; VEB Deutscher Verlag der
  Wissenschaften, Berlin, 1991.

\bibitem{Whitney1931}
H.~Whitney.
\newblock A theorem on graphs.
\newblock {\em Ann. of Math. (2)}, 32(2):378--390, 1931.

\bibitem{WY2023}
M.~C. Wigal and X.~Yu.
\newblock Tutte paths and long cycles in circuit graphs.
\newblock {\em J. Combin. Theory Ser. B}, 158(part 1):313--330, 2023.

\bibitem{Zhang1987}
C.~Q. Zhang.
\newblock Longest cycles and their chords.
\newblock {\em J. Graph Theory}, 11(4):521--529, 1987.

\end{thebibliography}

\end{document}